\documentclass[a4paper,12pt, abstract=true]{scrartcl}

\usepackage{german}
\usepackage[ansinew]{inputenc}
\usepackage[german,english,french,USenglish]{babel}
\usepackage{latexsym,amssymb,amsfonts,bbm,amsmath} 
\usepackage{theorem}
\usepackage{xcolor}
\usepackage{graphicx}
\usepackage{esint}
\usepackage{stmaryrd}
\usepackage[normalem]{ulem} 
\newtheorem{dummytheorem}{Dummy-Theorem}[section]
\newcommand{\proofendsign}{$\Box$} 
\newtheorem{definition}[dummytheorem]{Definition}
\newtheorem{lemma}[dummytheorem]{Lemma}
\newtheorem{theorem}[dummytheorem]{Theorem}
\newtheorem{proposition}[dummytheorem]{Proposition}

\newtheorem{corollary}[dummytheorem]{Corollary}

\newenvironment{proof}{{\noindent \bf Proof }}
 {{\hspace*{\fill}\proofendsign\par\bigskip}}

\theorembodyfont{\normalfont}
\newtheorem{remarknorm}[dummytheorem]{Remark}
\newtheorem{examplenorm}[dummytheorem]{Example}

\newcommand{\V}{\mathbf{V}}

\newcommand{\bE}{\mathbf{E}}
\newcommand{\bD}{\mathbf{D}}
\newcommand{\bC}{\mathbf{C}}

\newcommand{\bK}{\mathbf{K}}

\newcommand{\N}{\mathbb{N}}

\newcommand{\R}{\mathbb{R}}

\newcommand{\E}{\mathbb{E}}

\newcommand{\F}{\mathbf{F}}

\newcommand{\pr}{\mathbb{P}}
\newcommand{\ex}{\mathbb{E}}
\newcommand{\vari}{\mathbb{V}{\rm ar}}
\newcommand{\covi}{\mathbb{C}{\rm ov}}

\newcommand{\eins}{\mathbbm{1}}

\newcommand{\avatr}{{\rm AV@R}}

\newcommand{\cadlag}{c\`adl\`ag}

\allowdisplaybreaks

\begin{document}


\title{Functional delta-method for the bootstrap of uniformly quasi-Hadamard differentiable functionals}

\author{
Eric Beutner\footnote{Department of Quant.~Economics, Maastricht University, {\tt e.beutner@maastrichtuniversity.nl}}
\qquad\qquad
Henryk Zähle\footnote{Department of Mathematics, Saarland University, {\tt zaehle@math.uni-sb.de}}}
\date{} 
\maketitle

\begin{abstract}
The functional delta-method provides a convenient tool for deriving bootstrap consistency of a sequence of plug-in estimators w.r.t.\ a given functional from bootstrap consistency of the underlying sequence of estimators. It has recently been shown in \cite{BeutnerZaehle2016} that the range of applications of the functional delta-method for establishing bootstrap consistency {\em in probability} of the sequence of plug-in estimators can be considerably enlarged by replacing the usual condition of Hadamard differentiability of the given functional by the weaker condition of quasi-Hadamard differentiability. Here we introduce the notion of uniform quasi-Hadamard differentiability and show that this notion extends the set of functionals for which {\em almost sure} bootstrap consistency of the corresponding sequence of plug-in estimators can be obtained by the functional delta-method. We illustrate the benefit of our results by means of the Average Value at Risk functional as well as the composition of the Average Value at Risk functional and the compound convolution functional. For the latter we use a chain rule to be proved here. In our examples we consider the weighted exchangeable bootstrap for independent observations and the blockwise bootstrap for $\beta$-mixing observations.
\end{abstract}


\bigskip
\bigskip



\newpage

\section{Introduction}\label{section introduction}

The functional delta-method is a widely used technique to derive bootstrap consistency for a sequence of plug-in estimators w.r.t.\ a map $H$ from bootstrap consistency of the underlying sequence of estimators. An essential limitation of the classical functional delta-method for proving bootstrap consistency in probability (or outer probability) is the condition of Hadamard differentiability on $H$ (cf.\ Theorem 3.9.11 of \cite{van der Vaart Wellner 1996}). It is commonly acknowledged that Hadamard differentiability fails for many relevant maps $H$. Recently, it was demonstrated in \cite{BeutnerZaehle2016} that a functional delta-method for the bootstrap {\em in probability} can also be proved for {\em quasi}-Hadamard differentiable maps $H$. Quasi-Hadamard differentiability is a weaker notion of ``differentiability'' than Hadamard differentiability and can be obtained for many relevant statistical functionals $H$; see, e.g., \cite{BeutnerWuZaehle2012,BeutnerZaehle2010,BeutnerZaehle2012,Kraetschmeretal2013,KraetschmerZaehlel2016}. Using the classical functional delta-method to prove almost sure (or outer almost sure) bootstrap consistency for a sequence of plug-in estimators w.r.t.\ a map $H$ from almost sure (or outer almost sure) bootstrap consistency of the underlying sequence of estimators requires {\em uniform} Hadamard differentiability on $H$ (cf.\ Theorem 3.9.11 of \cite{van der Vaart Wellner 1996}). In the present article we will introduce the notion of {\em uniform quasi}-Hadamard differentiability and demonstrate that one can even obtain a functional delta-method for the {\em almost sure} bootstrap and {\em uniformly quasi}-Hadamard differentiable maps $H$. Proposition \ref{proposition AVaR} below shows that the notion of uniform quasi-Hadamard differentiability is weaker than uniform Hadamard differentiability, because this proposition shows that the Average Value at Risk functional, which fails to be Hadamard differentiable, is uniformly quasi-Hadamard differentiable.

To explain the background and the contribution of the paper at hand more precisely, assume that we are given an estimator $\widehat T_n$ for a parameter $\theta$ in a vector space, with $n$ denoting the sample size, and that we are actually interested in the aspect $H(\theta)$ of $\theta$. Here $H$ is any map taking values in a vector space. Then $H(\widehat T_n)$ is often a reasonable estimator for $H(\theta)$. One of the main objects in statistical inference is the distribution of the error $H(\widehat T_n)-H(\theta)$, because the error distribution can theoretically be used to derive confidence regions for $H(\theta)$. However in applications the exact specification of the error distribution is often hardly possible or even impossible. A widely used way out is to derive the {\em asymptotic} error distribution, i.e.\ the weak limit $\mu$ of ${\rm law}\{a_n(H(\widehat T_n)-H(\theta))\}$ for suitable normalizing constants $a_n$ tending to infinity, and to use $\mu$ as an approximation for $\mu_n:={\rm law}\{a_n(H(\widehat T_n)-H(\theta))\}$ for large $n$. Since $\mu$ usually still depends on the unknown parameter $\theta$, one should use the notation $\mu_\theta$ instead of $\mu$. In particular, one actually uses $\mu_{\widehat T_n}:=\mu_\theta|_{\theta=\widehat T_n}$ as an approximation for $\mu_n$ for large $n$.

Not least because of the estimation of the parameter $\theta$ of $\mu_\theta$, the approximation of $\mu_n$ by $\mu_{\widehat T_n}$ is typically only moderate. An often more efficient alternative technique to approximate $\mu_n$ is the bootstrap. The bootstrap has been introduced by Efron \cite{Efron1979} in 1979 and many variants of his method have been introduced since then. One may refer to \cite{DavisonHinkley1997,Efron1994,Lahiri2003,ShaoTu1995} for general accounts on this topic. The basic idea of the bootstrap is the following. Re-sampling the original sample according to a certain re-sampling mechanism (depending on the particular bootstrap method) one can sometimes construct a so-called bootstrap version $\widehat T_n^*$ of $\widehat T_n$ for which the conditional law of $a_n(H(\widehat T_n^*)-H(\widehat T_n))$ ``given the sample'' has the same weak limit $\mu_\theta$ as the law of $a_n(H(\widehat T_n)-H(\theta))$ has. The latter is referred to as bootstrap consistency. Since $\widehat T_n^*$ depends only on the sample and the re-sampling mechanism, one can at least numerically determine the conditional law of $a_n(H(\widehat T_n^*)-H(\widehat T_n))$ ``given the sample'' by means of a Monte Carlo simulation based on $L\gg n$ repetitions. The resulting law $\mu_L^*$ can then be used as an approximation of $\mu_{n}$, at least for large $n$.

In applications the roles of $\theta$ and $\widehat T_n$ are often played by a distribution function $F$ and the empirical distribution function $\widehat F_n$ of $n$ random variables that are identically distributed according to $F$, respectively. Not least for this particular setting several results on bootstrap consistency for $\widehat T_n$ are known (see also Section \ref{application to statistical functionals}). The functional delta-method then ensures that bootstrap consistency also holds for $H(\widehat T_n)$ when $H$ is suitably differentiable at $\theta$. Technically speaking, as indicated above, one has to distinguish between two types of bootstrap consistency. First bootstrap consistency {\em in probability} for $H(\widehat T_n)$ can be associated with
\begin{equation}\label{intro prob bootstrap}
    \lim_{n\to\infty}\pr^{\scriptsize{\sf out}}\big[\big\{\omega\in\Omega:\,d_{\scriptsize{\rm BL}}^\circ(P_n(\omega,\,\cdot\,),\mu_\theta)\ge\delta\big\}\big]=\,0\quad\mbox{ for all }\delta>0,
\end{equation}
where $\omega$ represents the sample, $P_n(\omega,\cdot)$ denotes the conditional law of $a_n(H(\widehat T_n^*)-H(\widehat T_n))$ given the sample $\omega$, $d_{\scriptsize{\rm BL}}^\circ$ is the bounded Lipschitz distance, and the superscript $^{\scriptsize{\sf out}}$ refers to outer probability. At this point it is worth pointing out that we consider weak convergence (resp.\ convergence in distribution) w.r.t.\ the open-ball $\sigma$-algebra, in symbols $\Rightarrow^\circ$ (resp.\ $\leadsto^\circ$), as defined in \cite[Section 6]{Billingsley1999} (see also \cite{Dudley1966,Dudley1967,Pollard1984,ShorackWellner1986}) and that by the Portmanteau theorem A.3 in \cite{BeutnerZaehle2016} weak convergence $\mu_n\Rightarrow^\circ\mu$ holds if and only if $d_{\scriptsize{\rm BL}}^\circ(\mu_n,\mu)\to 0$. Second bootstrap consistency {\em almost surely} for $H(\widehat T_n)$ means that
\begin{equation}\label{intro as bootstrap}
    {\rm law}\big\{a_n\big(H(\widehat T_n^*(\omega,\,\cdot\,))-H(\widehat T_n(\omega))\big)\big\}\,\Rightarrow^\circ\,\mu_\theta\qquad\mbox{$\pr$-a.e.\ $\omega$}.
\end{equation}
In \cite{BeutnerZaehle2016} it has been shown that (\ref{intro prob bootstrap}) follows from the respective analogue for $\widehat T_n$ when $H$ is suitably quasi-Hadamard differentiable at $\theta$. This extends Theorem 3.9.11 of \cite{van der Vaart Wellner 1996} which covers only Hadamard differentiable maps. In this article we will show that (\ref{intro as bootstrap}) follows from the respective analogue for $\widehat T_n$ when $H$ is suitably {\em uniformly} quasi-Hadamard differentiable at $\theta$; the notion of uniform quasi-Hadamard differentiable will be introduced in Definition \ref{definition quasi hadamard} below. This extends Theorem 3.9.13 of \cite{van der Vaart Wellner 1996} which covers only Hadamard differentiable maps.

To demonstrate that the theory presented here leads directly to new results for interesting applications we consider the Average Value at Risk functional and the compound distribution functional. To the best of our knowledge so far there do not exit results on almost sure bootstrap consistency for the Average Value at Risk functional when the underlying data are dependent. The same seems to be true for the compound distribution functional and consequently also for the composition of the Average Value at Risk functional and the compound distribution functional.

The rest of the article is organized as follows. In Section \ref{Abstract delta-method for the bootstrap} we introduce the definition of uniform quasi-Hadamard differentiability and prove a functional delta-method for almost sure bootstrap consistency based on it. In Section \ref{application to statistical functionals} this functional delta-method is discussed if the underlying sequence of estimators is the empirical distribution function. Section \ref{Examples} shows that the Average Value at Risk functional and the compound distribution functional are uniformly quasi-Hadamard differentiable. Moreover, we show there using a chain rule that the composition of the Average Value at Risk functional and the compound distribution functional is uniformly quasi-Hadamard differentiable. This chain rule is proved in the Appendix \ref{Delta-method for uniformly QHD maps} where we also prove a delta-method for uniformly quasi-Hadamard differentiable maps that is the basis for the main result of Section \ref{Abstract delta-method for the bootstrap}. In the Appendix \ref{Slutzky type results} we give results on convergence in distribution for the open-ball $\sigma$-algebra which are needed for the main results.


\section{Abstract delta-method for the bootstrap}\label{Abstract delta-method for the bootstrap}

Theorem \ref{modified delta method for the bootstrap} below provides an abstract delta-method for the almost sure bootstrap. It is based on the notion of uniform quasi-Hadamard differentiability which we introduce first. This sort of differentiability extends the notion of quasi-Hadamard differentiability as introduced in \cite{BeutnerZaehle2010,BeutnerZaehle2016}. The latter corresponds to the differentiability concept in (i) of Definition \ref{definition quasi hadamard} ahead with ${\cal S}$ and $\widetilde\bE$ as in (iii) and (v) of this definition. Let $\V$ and $\widetilde\V$ be vector spaces. Let $\bE\subseteq\V$ and $\widetilde\bE\subseteq\widetilde\V$ be subspaces equipped with norms $\|\cdot\|_{\bE}$ and $\|\cdot\|_{\widetilde\bE}$, respectively. Let
$$
    H:\V_H\longrightarrow\widetilde\V
$$
be any map defined on some subset $\V_H\subseteq\V$.

\begin{definition}\label{definition quasi hadamard}
Let $\bE_0$ be a subset of $\bE$, and ${\cal S}$ be a set of sequences in $\V_H$.

(i) The map $H$ is said to be uniformly quasi-Hadamard differentiable w.r.t.\ ${\cal S}$ tangentially to $\bE_0\langle\bE\rangle$ with trace $\widetilde\bE$ if $H(y_1)-H(y_2)\in\widetilde\bE$ for all $y_1,y_2\in\V_H$, $n\in\N$, and there is some continuous map $\dot H_{\cal S}:\bE_0\rightarrow\widetilde\bE$ such that
\begin{eqnarray}\label{def eq for AHD}
    \lim_{n\to\infty}\Big\| \dot H_{\cal S}(x)-\frac{H(\theta_{n}+\varepsilon_nx_{n})-H(\theta_n)}{\varepsilon_n}\Big\|_{\widetilde\bE}=0
\end{eqnarray}
holds for each quadruple $((\theta_n),x,(x_{n}),(\varepsilon_n))$, with $(\theta_{n})\in{\cal S}$, $x\in\bE_0$, $(x_{n})\subseteq\bE$ satisfying $\|x_{n}-x\|_{\bE}\to 0$ as well as $(\theta_n+\varepsilon_nx_{n})\subseteq\V_H$, and $(\varepsilon_n)\subseteq(0,\infty)$ satisfying $\varepsilon_n\to 0$. In this case the map $\dot H_{\cal S}$ is called uniform quasi-Hadamard derivative of $H$ w.r.t.\ ${\cal S}$ tangentially to $\bE_0\langle\bE\rangle$.

(ii) If ${\cal S}$ consists of all sequences $(\theta_n)\subseteq\V_H$ with $\theta_n-\theta\in\bE$, $n\in\N$, and $\|\theta_n-\theta\|_\bE\to 0$ for some fixed $\theta\in\V_H$, then we replace the phrase ``w.r.t.\ ${\cal S}$'' by ``at $\theta$'' and ``$\dot H_{\cal S}$'' by ``$\dot H_\theta$''.

(iii) If ${\cal S}$ consists only of the constant sequence $\theta_n=\theta$, $n\in\N$, then we skip the phrase ``uniformly'' and replace the phrase ``w.r.t.\ ${\cal S}$'' by ``at $\theta$'' and ``$\dot H_{\cal S}$'' by ``$\dot H_\theta$''. In this case we may also replace ``$H(y_1)-H(y_2)\in\widetilde\bE$ for all $y_1,y_2\in\V_H$'' by ``$H(y) - H(\theta) \in\widetilde\bE$ for all $y\in\V_H$''.

(iv) If $\bE=\V$, then we skip the phrase ``quasi-''.

(v) If $\widetilde\bE=\widetilde\V$, then we skip the phrase ``with trace $\widetilde\bE$''.
\end{definition}

The conventional notion of uniform Hadamard differentiability as used in Theorem 3.9.11 of \cite{van der Vaart Wellner 1996} corresponds to the differentiability concept in (i) with ${\cal S}$ as in (ii), $\bE$ as in (iv), and $\widetilde\bE$ as in (v). Proposition \ref{proposition AVaR} below shows that it is beneficial to refrain from insisting on $\bE=\V$ as in (iv). It was recently discussed in \cite{Bellonietal2016} that it can be also beneficial to refrain from insisting on the assumption of (ii). For $\bE=\V$ (``non-quasi'' case) uniform Hadamard differentiability in the sense of Definition B.1 in \cite{Bellonietal2016} corresponds to uniform Hadamard differentiability in the sense of our Definition \ref{definition quasi hadamard} (part (i) and (iv)) when ${\cal S}$ is chosen as the set of all sequences $(\theta_n)$ in a compact metric space $(\bK_\theta,d_\bK)$ with $\theta\in\bK_\theta\subseteq\V_H$ for which $d_\bK(\theta_n,\theta)\to0$. In Comment B.3 of \cite{Bellonietal2016} it is illustrated by means of the quantile functional that this notion of differentiability (subject to a suitable choice of $(\bK_\theta,d_\bK)$) is strictly weaker than the notion of uniform Hadamard differentiability that was used in the classical delta-method for the almost sure bootstrap, Theorem 3.9.11 in \cite{van der Vaart Wellner 1996}. Although this shows that the flexibility w.r.t.\ ${\cal S}$ in our Definition \ref{definition quasi hadamard} can be beneficial, it is somehow even more important that we allow for the ``quasi'' case.

Of course, the smaller the family ${\cal S}$ the weaker the condition of uniform quasi-Hadamard differentiability w.r.t.\ ${\cal S}$. On the other hand, if the set ${\cal S}$ is too small then condition (e) in Theorem \ref{modified delta method for the bootstrap} ahead may fail. That is, for an application of the functional delta-method in the form of Theorem \ref{modified delta method for the bootstrap} the set ${\cal S}$ should be large enough for condition (e) to be fulfilled and small enough for being able to establish uniform quasi-Hadamard differentiability w.r.t.\ ${\cal S}$ of the map $H$.

We now turn to the abstract delta-method. As mentioned in the introduction, convergence in distribution will always be considered for the open-ball $\sigma$-algebra. We will use the terminology {\em convergence in distribution$^\circ$} (symbolically $\leadsto^\circ$) for this sort of convergence; for details see the Appendix \ref{Appendix Weak Convergence} and the Appendices A--C of \cite{BeutnerZaehle2016}. In a separable metric space the notion of convergence in distribution$^\circ$ boils down to the conventional notion of convergence in distribution for the Borel $\sigma$-algebra. In this case we use the symbol $\leadsto$ instead of $\leadsto^\circ$.

Let $(\Omega,{\cal F},\pr)$ be a probability space, and $(\widehat T_n)$ be a sequence of maps
$$
    \widehat T_n:\Omega\longrightarrow\V.
$$
Regard $\omega\in\Omega$ as a sample drawn from $\pr$, and $\widehat T_n(\omega)$ as a statistic derived from $\omega$. Somewhat unconventionally, we do not (need to) require at this point that $\widehat T_n$ is measurable w.r.t.\ any $\sigma$-algebra on $\V$. Let $(\Omega',{\cal F}',\pr')$ be another probability space and set
$$
    (\overline\Omega,\overline{\cal F},\overline{\pr}):=(\Omega\times\Omega',{\cal F}\otimes{\cal F}',\pr\otimes\pr').
$$
The probability measure $\pr'$ represents a random experiment that is run independently of the random sample mechanism $\pr$. In the sequel, $\widehat T_n$ will frequently be regarded as a map defined on the extension $\overline\Omega$ of $\Omega$. Let
$$
    \widehat T_n^*: \overline\Omega\longrightarrow\V
$$
be any map. Since $\widehat T_n^{*}(\omega,\omega')$ depends on both the original sample $\omega$ and the outcome $\omega'$ of the additional independent random experiment, we may regard $\widehat T_n^*$ as a bootstrapped version of $\widehat T_n$. Moreover, let
$$
    \widehat C_n: \Omega \longrightarrow \V
$$
be any map. As with $\widehat T_n$ we often regard $\widehat C_n$ as a map defined on the extension $\overline\Omega$ of $\Omega$. We will use $\widehat C_n$ together with a scaling sequence to get weak convergence results for $\widehat T_n^*$. The role of $\widehat C_n$ is often played by $\widehat T_n$ itself (cf.\ Example \ref{iid bootstrap}), but sometimes also by a different map (cf.\ Example \ref{block bootstrap}). Assume that $\widehat T_n$, $\widehat T_n^*$, and $\widehat C_n$ take values only in $\V_H$.

Let ${\cal B}^\circ$ and $\widetilde{\cal B}^\circ$ be the open-ball $\sigma$-algebras on $\bE$ and $\widetilde\bE$ w.r.t.\ the norms $\|\cdot\|_\bE$ and $\|\cdot\|_{\widetilde\bE}$, respectively. Note that ${\cal B}^\circ$ coincides with the Borel $\sigma$-algebra on $\bE$ when $(\bE,\|\cdot\|_{\bE})$ is separable. The same is true for $\widetilde{\cal B}^\circ$. Set $\overline{\widetilde\bE}:=\widetilde\bE\times\widetilde\bE$ and let $\overline{\widetilde{\cal B}^\circ}$ be the $\sigma$-algebra on $\overline{\widetilde\bE}$ generated by the open balls w.r.t.\ the metric $\overline{\widetilde d}((\widetilde x_1,\widetilde x_2),(\widetilde y_1,\widetilde y_2)):=\max\{\|\widetilde x_1-\widetilde y_1\|_{\widetilde\bE};\|\widetilde x_2-\widetilde y_2\|_{\widetilde\bE}\}$. Recall that $\overline{\widetilde{\cal B}^\circ}\subseteq\widetilde{\cal B}^\circ\otimes\widetilde{\cal B}^\circ$, because any $\overline{\widetilde{d}}$-open ball in $\overline{\widetilde{\bE}}$ is the product of two $\|\cdot\|_{\widetilde\bE}$-open balls in $\widetilde\bE$.

The following Theorem \ref{modified delta method for non bootstrap} is a consequence of Theorem \ref{modified delta method} in the Appendix \ref{Delta-method for uniformly QHD maps} as we assume that $\widehat T_n$ takes values only in $\V_H$. The proof of the measurability statement of Theorem \ref{modified delta method for non bootstrap} is given in the proof of Theorem \ref{modified delta method for the bootstrap}. Theorem \ref{modified delta method for non bootstrap} is stated here because, together with Theorem \ref{modified delta method for the bootstrap}, it implies almost sure bootstrap consistency whenever the limit $\xi$ is the same in Theorem \ref{modified delta method for non bootstrap} and Theorem \ref{modified delta method for the bootstrap}.

\begin{theorem}\label{modified delta method for non bootstrap}
Let $(\theta_n)$ be a sequence in $\V_H$ and $\mathcal{S}:=\{(\theta_n)\}$. Let $\bE_0\subseteq\bE$ be a separable subspace and assume that $\bE_0\in {\cal B}^{\circ}$. Let $(a_n)$ be a sequence of positive real numbers with $a_n\to\infty$, and assume that the following assertions hold:
\begin{itemize}
    \item[(a)] $a_n(\widehat T_n-\theta_n)$ takes values only in $\bE$, is $({\cal F},{\cal B}^{\circ})$-measurable, and satisfies
        \begin{equation}\label{modified delta method for the bootstrap - assumption - 10}
             a_n(\widehat T_n-\theta_n)\,\leadsto^\circ\,\xi\qquad\mbox{in $(\bE,{\cal B}^{\circ},\|\cdot\|_{\bE})$}
        \end{equation}
    for some $(\bE,{\cal B}^{\circ})$-valued random variable $\xi$ on some probability space $(\Omega_0,{\cal F}_0,\pr_0)$ with $\xi(\Omega_0)\subseteq\bE_0$.
    \item[(b)] $a_n(H(\widehat T_n)-H(\theta_n))$ takes values only in $\widetilde\bE$ and is $({\cal F},\widetilde{\cal B}^\circ)$-measurable.
    \item[(c)] $H$ is uniformly quasi-Hadamard differentiable w.r.t.\ ${\cal S}$ tangentially to $\bE_0\langle\bE\rangle$ with trace $\widetilde\bE$ and uniform quasi-Hadamard derivative $\dot H_{\cal S}$.
    \end{itemize}
Then $\dot H_{\cal S}(\xi)$ is $({\cal F}_0,\widetilde{\cal B}^{\circ})$-measurable and
$$
    a_n\big(H(\widehat T_n)-H(\theta_n)\big)\,\leadsto^\circ\,\dot H_{\cal S}(\xi)\qquad\mbox{in $(\widetilde\bE,\widetilde{\cal B}^\circ,\|\cdot\|_{{\widetilde\bE}})$}.
$$
\end{theorem}

\begin{theorem}\label{modified delta method for the bootstrap}
Let ${\cal S}$ be any set of sequences in $\V_H$. Let $\bE_0\subseteq\bE$ be a separable subspace and assume that $\bE_0\in {\cal B}^{\circ}$. Let $(a_n)$ be a sequence of positive real numbers with $a_n\to\infty$, and assume that the following assertions hold:
\begin{itemize}
    \item[(a)] $a_n(\widehat T_n^{*} - \widehat C_n)$ takes values only in $\bE$, is $(\overline{\cal F},{\cal B}^\circ)$-measurable, and satisfies
    \begin{equation}\label{modified delta method for the bootstrap - assumption - 15}
             a_n\big(\widehat T_n^{*}(\omega,\,\cdot\,) - \widehat C_n(\omega)\big)\,\leadsto^\circ\,\xi\qquad\mbox{in $(\bE,\mathcal{B}^{\circ},\|\cdot\|_{\bE})$},\qquad\mbox{$\pr$-a.e.\ $\omega$}
    \end{equation}
  for some $(\bE,{\cal B}^{\circ})$-valued random variable $\xi$ on some probability space $(\Omega_0,{\cal F}_0,\pr_0)$ with $\xi(\Omega_0)\subseteq\bE_0$.
    \item[(b)] $a_n(H(\widehat T_{n}^*)-H(\widehat C_n))$ takes values only in $\widetilde\bE$ and is $(\overline{\cal F},\widetilde{\cal B}^\circ)$-measurable.
    \item[(c)] $H$ is uniformly quasi-Hadamard differentiable w.r.t.\ ${\cal S}$ tangentially to $\bE_0\langle\bE\rangle$ with trace $\widetilde\bE$ and uniform quasi-Hadamard derivative $\dot H_{\cal S}$.
    \item[(d)] The uniform quasi-Hadamard derivative $\dot H_{\cal S}$ can be extended from $\bE_0$ to $\bE$ such that the extension $\dot H_{\cal S}:\bE\rightarrow\widetilde{\bE}$ is $({\cal B}^{\circ},\widetilde{\cal B}^\circ)$-measurable and continuous at every point of $\bE_0$.
    \item[(e)] $(\widehat C_n(\omega))\in{\cal S}$ for $\pr$-a.e.\ $\omega$.
    \item[(f)] The map $h:\overline{\widetilde\bE}\rightarrow\widetilde\bE$ defined by $h(\widetilde x_1,\widetilde x_2):=\widetilde x_1-\widetilde x_2$ is $(\overline{\widetilde{\cal B}^\circ},\widetilde{\cal B}^\circ)$-measurable.
\end{itemize}
Then $\dot H_{\cal S}(\xi)$ is $({\cal F}_0,\widetilde{\cal B}^{\circ})$-measurable and
\begin{equation}\label{modified delta method for the bootstrap - assumption - 40}
    a_n\big(H(\widehat T_n^*(\omega,\,\cdot\,))-H(\widehat C_n(\omega))\big)\,\leadsto^\circ\,\dot H_{\cal S}(\xi)\qquad\mbox{in $(\widetilde\bE,\widetilde{\cal B}^\circ,\|\cdot\|_{{\widetilde\bE}})$},\qquad\mbox{$\pr$-a.e.\ $\omega$}.
\end{equation}
\end{theorem}

\begin{remarknorm}\label{modified delta method for the bootstrap - remark - 00}
In condition (a) of Theorem \ref{modified delta method for the bootstrap} it is assumed that $a_n(\widehat T_n^* - \widehat C_n)$ is $(\overline{\cal F},{\cal B}^\circ)$-measurable for $\overline{\cal F}:={\cal F}\otimes{\cal F}'$. Thus the mapping $\omega'\mapsto a_n(\widehat T_n^*(\omega,\omega')-\widehat C_n(\omega))$ is $({\cal F}',{\cal B}^\circ)$-measurable for every fixed $\omega\in\Omega$. That is, $a_n(\widehat T_n^*(\omega,\cdot) - \widehat C_n(\omega))$ can be seen as an $(\bE,{\cal B}^\circ)$-valued random variable on $(\Omega',{\cal F}',\pr')$ for every fixed $\omega\in\Omega$, so that assertion (\ref{modified delta method for the bootstrap - assumption - 15}) makes sense. By the same line of reasoning one can regard $a_n(H(\widehat T_n^*(\omega,\cdot))-H(\widehat C_n(\omega)))$ as an $(\widetilde\bE,\widetilde{\cal B}^\circ)$-valued random variable on $(\Omega',{\cal F}',\pr')$ for every fixed $\omega\in\Omega$, so that also assertion (\ref{modified delta method for the bootstrap - assumption - 40}) makes sense.
{\hspace*{\fill}$\Diamond$\par\bigskip}
\end{remarknorm}

\begin{remarknorm}
Condition (c) in Theorem \ref{modified delta method for non bootstrap} (resp.\ Theorem \ref{modified delta method for the bootstrap}) assumes that the trace is given by $\widetilde\bE$, which implies that the first part of condition (b) in Theorem \ref{modified delta method for non bootstrap} (resp.\ Theorem \ref{modified delta method for the bootstrap}) is automatically satisfied.
{\hspace*{\fill}$\Diamond$\par\bigskip}
\end{remarknorm}

\begin{remarknorm}\label{modified delta method for the bootstrap - remark}
Condition (f) of Theorem \ref{modified delta method for the bootstrap} is automatically fulfilled when $(\widetilde\bE,\|\cdot\|_{\widetilde\bE})$ is separable. Indeed, in this case we have $\overline{\widetilde{\cal B}^\circ}=\widetilde{\cal B}^\circ\otimes\widetilde{\cal B}^\circ$ so that every continuous map $h:\overline{\widetilde\bE}\rightarrow\widetilde \bE$ (such as $h(\widetilde x_1,\widetilde x_2):=\widetilde x_1-\widetilde x_2$) is $(\overline{\widetilde{\cal B}^\circ},\widetilde{\cal B}^\circ)$-measurable.
{\hspace*{\fill}$\Diamond$\par\bigskip}
\end{remarknorm}

\begin{proof}{\bf of Theorem \ref{modified delta method for the bootstrap}}
First note that by the assumption imposed on $\xi$ (cf.\ assumption (a)) and assumption (c) the map $\dot H_{\cal S}(\xi)$ is $(\mathcal{F}_{0},\widetilde{\cal B}^{\circ})$-measurable. Next note that
\begin{eqnarray*}
    \lefteqn{a_n\big(H(\widehat T_n^*(\omega,\omega'))-H(\widehat C_n(\omega))\big)}\\
    & = &  \big\{a_n\big(H(\widehat T_n^*(\omega,\omega'))-H(\widehat C_n(\omega))\big)-\dot H_{\cal S}\big(a_n(\widehat T_n^*(\omega,\omega')-\widehat C_n(\omega))\big)\big\}\\
    & &  +\,\dot H_{\cal S}\big(a_n(\widehat T_n^*(\omega,\omega') - \widehat C_n(\omega))\big)\\
    & =: & S_1(\omega,\omega')+S_2(\omega,\omega').
\end{eqnarray*}
By (\ref{modified delta method for the bootstrap - assumption - 15}) in assumption (a) and the Continuous Mapping theorem in the form of \cite[Theorem 6.4]{Billingsley1999} (along with $\pr_0\circ\xi^{-1}[\bE_0]=1$ and the continuity of $\dot H_{\cal S}$), we have that $S_2(\omega,\cdot)\leadsto^\circ\dot H_{\cal S}(\xi)$ for $\pr$-a.e.\ $\omega$. Moreover, for every fixed $\omega$ we have that $\omega'\mapsto S_1(\omega,\omega')$ is $({\cal F}',\widetilde{\cal B}^\circ)$-measurable by assumption (f), and for $\pr$-a.e.\ $\omega$ we have
$$
     a_n\big(H_n(\widehat T_n^*(\omega,\cdot))-H_n(\widehat C_n(\omega))\big)-\dot H_{\cal S}\big(a_n(\widehat T_n^*(\omega,\omega') - \widehat C_n(\omega))\big)\,\rightarrow^{{\sf p},\circ}\,0_{\widetilde\bE}
$$
by part (ii) of Theorem \ref{modified delta method} (recall that $\widehat T_n^*$ was assumed to take values only in $\V_H$), where $\rightarrow^{{\sf p},\circ}$ refers to convergence in probability$^\circ$ (cf.\ Section \ref{Slutzky type results}) and $\widehat T_n^*(\omega,\cdot)$, $\widehat C_n(\omega)$, $\{(\widehat C_n(\omega))\}$ play the roles of $\widehat T_n(\cdot)$, $\theta_n$, ${\cal S}$, respectively. Hence, from Corollary \ref{Slutzky corollary 2} we get that (\ref{modified delta method for the bootstrap - assumption - 40}) holds.
\end{proof}


\section{Application to plug-in estimators of statistical functionals}\label{application to statistical functionals}

Let $\bD$ be the space of all \cadlag\ functions $v$ on $\R$ with finite sup-norm $\|v\|_\infty:=\sup_{t\in\R}|v(t)|$, and ${\cal D}$ be the $\sigma$-algebra on $\bD$ generated by the one-dimensional coordinate projections $\pi_t$, $t\in\R$, given by $\pi_t(v):=v(t)$. Let $\phi:\R\rightarrow[1,\infty)$ be a weight function, i.e.\ a continuous function being non-increasing  on $(-\infty,0]$ and non-decreasing on $[0,\infty)$. Let $\bD_{\phi}$ be the subspace of $\bD$ consisting of all $x\in\bD$ satisfying $\|x\|_\phi:=\|x\phi\|_\infty<\infty$ and $\lim_{|t|\to\infty} |x(t)|=0$. The latter condition automatically holds when $\lim_{|t|\to\infty} \phi(t)=\infty$. Let ${\cal D}_{\phi}:={\cal D}\cap\bD_{\phi}$ be the trace $\sigma$-algebra on $\bD_{\phi}$. The $\sigma$-algebra on $\bD_{\phi}$ generated by the $\|\cdot\|_\phi$-open balls will be denoted by ${\cal B}_{\phi}^\circ$. Lemma 4.1 in \cite{BeutnerZaehle2016} shows that it coincides with ${\cal D}_{\phi}$.

Let $\bC_{\phi}\subseteq\bD_{\phi}$ be a $\|\cdot\|_\phi$-separable subspace and assume $\bC_{\phi}\in{\cal D}_{\phi}$. Moreover, let $H:\bD(H)\rightarrow\widetilde\V$ be a map defined on a set $\bD(H)$ of distribution functions of finite (not necessarily probability) Borel measures on $\R$, where $\widetilde\V$ is any vector space. In particular, $\bD(H)\subseteq\bD$. In the following, $\bD$, $(\bD_{\phi},{\cal B}_{\phi}^\circ,\|\cdot\|_\phi)$, $\bC_{\phi}$, and $\bD(H)$ will play the roles of $\V$, $(\bE,{\cal B}^{\circ},\|\cdot\|_{\bE})$, $\bE_0$, and $\V_H$, respectively. As before we let $(\widetilde\bE,\|\cdot\|_{\widetilde\bE})$ be a normed subspace of $\widetilde\V$ equipped with the corresponding open-ball $\sigma$-algebra $\widetilde{\cal B}^\circ$.

Let $(\Omega,{\cal F},\pr)$ be a probability space. Let $(F_n)\subseteq\bD(H)$ be any sequence and $(X_i)$ be a sequence of real-valued random variables on $(\Omega,{\cal F},\pr)$. Moreover let $\widehat F_n:\Omega\rightarrow\bD$ be the empirical distribution function of $X_1,\ldots,X_n$, which will play the role of $\widehat T_n$. It is defined by
\begin{equation}\label{Def EmpDF}
    \widehat F_n:=\frac{1}{n}\sum_{i=1}^n\eins_{[X_i,\infty)}.
\end{equation}
Assume that $\widehat F_n$ takes values only in $\bD(H)$. Let $(\Omega',{\cal F}',\pr')$ be another probability space and set $(\overline\Omega,\overline{\cal F},\overline{\pr}):=(\Omega\times\Omega',{\cal F}\otimes{\cal F}',\pr\otimes\pr')$. Moreover, let $\widehat F_n^*:\overline{\Omega}\rightarrow\bD$ be any map. Assume that $\widehat F_n^*$ take values only in $\bD(H)$. Furthermore, let $\widehat C_n: \Omega \rightarrow \bD$ be any map that takes values only in $\bD(H)$. In the present setting Theorems \ref{modified delta method for non bootstrap} and \ref{modified delta method for the bootstrap} can be reformulated as follows, where we recall from Remark \ref{modified delta method for the bootstrap - remark} that condition (f) of Theorem \ref{modified delta method for the bootstrap} is automatically fulfilled when $(\widetilde\bE,\|\cdot\|_{\widetilde\bE})$ is separable.

\begin{corollary}\label{modified delta method for non bootstrap - II}
Let $(F_n)$  be a sequence in $\bD(H)$ and ${\cal S}:=\{(F_n)\}$. Let $(a_n)$ be a sequence of positive real numbers with $a_n\to\infty$, and assume that the following assertions hold:
\begin{itemize}
    \item[(a)] $a_n(\widehat F_n-F_n)$ takes values only in $\bD_{\phi}$ and satisfies 
    \begin{equation}\label{modified delta method for the bootstrap - assumption - 10 - II}
             a_n(\widehat F_n-F_n)\,\leadsto^\circ\,B\qquad\mbox{in $(\bD_{\phi},{\cal B}_{\phi}^\circ,\|\cdot\|_{\phi})$}
    \end{equation}
    for some $(\bD_{\phi},{\cal B}_{\phi}^\circ)$-valued random variable $B$ on some probability space $(\Omega_0,{\cal F}_0,\pr_0)$ with $B(\Omega_0)\subseteq\bC_{\phi}$.
    \item[(b)] $a_n(H(\widehat F_n)-H(F_n))$ takes values only in $\widetilde\bE$ and is $({\cal F},\widetilde{\cal B}^\circ)$-measurable.
    \item[(c)] $H$ is uniformly quasi-Hadamard differentiable w.r.t.\ ${\cal S}$ tangentially to $\bC_{\phi}\langle\bD_{\phi}\rangle$ with trace $\widetilde\bE$ and uniform quasi-Hadamard derivative $\dot H_{\cal S}$.
\end{itemize}
Then $\dot H_{\cal S}(B)$ is $({\cal F}_0,\widetilde{\cal B}^\circ)$-measurable and
$$
    a_n\big(H(\widehat F_n)-H(F_n)\big)\,\leadsto^\circ\,\dot H_{\cal S}(B)\qquad\mbox{in $(\widetilde\bE,\widetilde{\cal B}^\circ,\|\cdot\|_{{\widetilde\bE}})$}.
$$
\end{corollary}

Note that the measurability assumption in condition (a) of Theorem \ref{modified delta method for non bootstrap} is automatically satisfied in the present setting (and is therefore omitted in condition (a) of Corollary \ref{modified delta method for non bootstrap - II}). Indeed, $a_n(\widehat F_n-F)$ is $({\cal F},{\cal B}_{\phi}^\circ)$-measurable, because it is easily seen to be $({\cal F},{\cal D}_{\phi})$-measurable and we have noted above that ${\cal B}_{\phi}^\circ={\cal D}_{\phi}$.

\begin{corollary}\label{modified delta method for the bootstrap - II}
Let ${\cal S}$ be any set of sequences in $\bD(H)$. Let $(a_n)$ be a sequence of positive real numbers with $a_n\to\infty$, and assume that the following assertions hold:
\begin{itemize}
    \item[(a)] $a_n(\widehat F_n^{*} - \widehat C_n)$ takes values only in $\bD_{\phi}$, is $(\overline{\cal F},{\cal B}_{\phi}^\circ)$-measurable, and
    \begin{equation}\label{modified delta method for the bootstrap - assumption - 15 - II}
            a_n\big(\widehat F_n^{*}(\omega,\cdot)-\widehat C_n(\omega)\big)\,\leadsto^\circ\,B\qquad\mbox{in $(\bD_{\phi},{\cal B}_{\phi}^\circ,\|\cdot\|_{\phi})$},\qquad\mbox{$\pr$-a.e.\ $\omega$}
    \end{equation}
    for some $(\bD_{\phi},{\cal B}_{\phi}^\circ)$-valued random variable $B$ on some probability space $(\Omega_0,{\cal F}_0,\pr_0)$ with $B(\Omega_0)\subseteq\bC_{\phi}$.
    \item[(b)] $a_n(H(\widehat F_{n}^*)-H(\widehat C_n))$ takes values only in $\widetilde\bE$ and is $(\overline{\cal F},\widetilde{\cal B}^\circ)$-measurable.
    \item[(c)] $H$ is uniformly quasi-Hadamard differentiable w.r.t.\ ${\cal S}$ tangentially to $\bC_{\phi}\langle\bD_{\phi}\rangle$ with trace $\widetilde\bE$ and uniform quasi-Hadamard derivative $\dot H_{\cal S}$.
    \item[(d)] The uniform quasi-Hadamard derivative $\dot H_{\cal S}$ can be extended from $\bC_{\phi}$ to $\bD_{\phi}$ such that the extension $\dot H_{\cal S}:\bD_{\phi}\rightarrow\widetilde\bE$ is $({\cal B}_{\phi}^\circ,\widetilde{\cal B}^\circ)$-measurable, and continuous at every point of $\bC_{\phi}$.
    \item[(e)] $(\widehat C_n(\omega))\in{\cal S}$ for $\pr$-a.e.\ $\omega$.
    \item[(f)] The map $h:\overline{\widetilde\bE}\rightarrow\widetilde\bE$ defined by $h(\widetilde x_1,\widetilde x_2):=\widetilde x_1-\widetilde x_2$ is $(\overline{\widetilde{\cal B}^\circ},\widetilde{\cal B}^\circ)$-measurable.
\end{itemize}
Then $\dot H_{\cal S}(B)$ is $({\cal F}_0,\widetilde{\cal B}^\circ)$-measurable and
$$
    a_n\big(H(\widehat F_n^*(\omega,\cdot))-H(\widehat C_n(\omega))\big)\,\leadsto^\circ \,\dot H_{\cal S}(B)\qquad\mbox{in $(\widetilde\bE,\widetilde{\cal B}^\circ,\|\cdot\|_{{\widetilde\bE}})$},\qquad\mbox{$\pr$-a.e.\ $\omega$}.
$$
\end{corollary}

The following two examples illustrate $\widehat F_n^*$ and $\widehat C_n$. In S1.\ and S2.\ in the first example, i.e.\ Example \ref{iid bootstrap}, we have $\widehat C_n=\widehat{F}_n$, and in S3.\ of this example as well as in the second example, i.e.~Example \ref{block bootstrap}, $\widehat C_n$ may differ from $\widehat{F}_n$. Examples for uniformly quasi-Hadamard differentiable functionals $H$ can be found in Section \ref{Examples}. In the examples in Sections \ref{Average Value at Risk functional} and \ref{Composition functional} we have $\widetilde\V=\widetilde\bE=\R$, and in the Example in Section \ref{Convolution functional} we have $\widetilde\V=\bD$ and $\widetilde\bE=\bD_\phi$ for some $\phi$.

\begin{examplenorm}\label{iid bootstrap}
Let $(X_i)$ be a sequence of i.i.d.\ real-valued random variables on $(\Omega,{\cal F},\pr)$ with distribution function $F$, and $\widehat F_n$ be given by (\ref{Def EmpDF}). Let $(W_{ni})$ be a triangular array of nonnegative real-valued random variables on $(\Omega',{\cal F}',\pr')$ such that $(W_{n1},\ldots,W_{nn})$ is an exchangeable random vector for every $n\in\N$, and define the map $\widehat F_n^*:\overline\Omega\rightarrow\bD$ by
\begin{equation}\label{Def F star}
    \widehat F_n^*(\omega,\omega'):=\frac{1}{n} \sum_{i=1}^n W_{ni} (\omega')\,\eins_{[X_i(\omega),\infty)}.
\end{equation}
Note that the sequence $(X_i)$ and the triangular array $(W_{ni})$ regarded as families of random variables on the product space $(\overline\Omega,\overline{\cal F},\overline\pr):=(\Omega\times\Omega',{\cal F}\otimes{\cal F}',\pr\otimes\pr')$ are independent. Of course, we will tacitly assume that $(\Omega',{\cal F}',\pr')$ is rich enough to host all the random variables used below. Similar as in Lemma 5.1 of \cite{BeutnerZaehle2016} it can be shown that $a_n(\widehat F_n^*-\widehat{C}_n)$, with $\widehat{C}_n:= \overline{W}_n \widehat{F}_n$, takes values only in $\bD_{\phi}$ and is $(\overline{\cal F},{\cal D}_{\phi})$-measurable, where $\overline{W}_n:=\frac{1}{n}\sum_{i=1}^nW_{ni}$. That is, the first part of condition (a) of Corollary \ref{modified delta method for the bootstrap - II} holds true for $\widehat C_n:=\overline{W}_n\widehat F_n$. Now assume that $F$ satisfies $\int\phi^2dF<\infty$ and that the following three assertions hold.
\begin{itemize}
    \item[A1.] $\sup_{n\in\N}\int_0^\infty\pr'[|W_{n1}-\overline{W}_n|>t]^{1/2}\,dt<\infty$.
    \item[A2.] $\frac{1}{\sqrt{n}}\,\ex'\big[\max_{1\le i\le n}|W_{ni}-\overline{W_n}|\big]\rightarrow 0$.
    \item[A3.] $\frac{1}{n}\sum_{i=1}^n(W_{ni}-\overline{W}_n)^2\rightarrow 1$ in $\pr'$-probability.
\end{itemize}
Then, arguing as in Example 4.3 and Section 5.1 of \cite{BeutnerZaehle2016}, results in \cite{ShorackWellner1986} and \cite{van der Vaart Wellner 1996} imply that respectively condition (a) of Corollary \ref{modified delta method for non bootstrap - II} (with $F_n:=F$) and the second part of condition (a) of Corollary \ref{modified delta method for the bootstrap - II} (with $\widehat C_n:=\overline{W}_n\widehat F_n$) hold for $a_n:=\sqrt{n}$ and $B:=B_F$, where $B_F$ is an $F$-Brownian bridge, i.e.\ a centered Gaussian process with covariance function $\Gamma(t_0,t_1)=F(t_0\wedge t_1)\overline{F}(t_0\vee t_1)$. Here $\bC_{\phi}$ can be chosen to be the set $\bC_{\phi,F}$ of all $v\in\bD_\phi$ whose discontinuities are also discontinuities of $F$.

Examples 3.6.9, 3.6.10, and 3.6.12 in \cite{van der Vaart Wellner 1996} show that conditions A1.--A3.\ are satisfied if one of the following three specific settings is met:
\begin{itemize}
    \item[S1.] 
    The random vector $(W_{n1},\ldots,W_{nn})$ is multinomially distributed according to the parameters $n$ and $p_1=\cdots=p_n=\frac{1}{n}$ for every $n\in\N$.
    \item[S2.] 
    $W_{ni}=Y_i/\overline{Y}_n$ for every $i=1,\ldots,n$ and $n\in\N$, where $\overline{Y}_n:=\frac{1}{n}\sum_{j=1}^nY_j$ and $(Y_j)$ is any sequence of nonnegative i.i.d.\ random variables on $(\Omega',{\cal F}',\pr')$ with $\int_0^\infty\pr'[Y_{1}>t]^{1/2}\,dt<\infty$ and $\vari'[Y_1]^{1/2}=\ex'[Y_1]>0$.
    \item[S3.] 
    $W_{ni}=Y_i$ for every $i=1,\ldots,n$ and $n\in\N$, where $(Y_i)$ is any sequence of non-negative i.i.d.\ random variables on $(\Omega',{\cal F}',\pr')$ with $\int_0^\infty\pr'[Y_{1}>t]^{1/2}\,dt<\infty$ and $\vari'[Y_{1}]=1$.
\end{itemize}
Setting S1.\ is nothing but {\em Efron's boostrap} \cite{Efron1979} and Setting S3.\ is sometimes referred to as {\em wild bootstrap}. If in Setting S2.\ the distribution of $Z_1$ is the exponential distribution with parameter $1$, then the resulting scheme is in line with the {\em Bayesian bootstrap} of Rubin \cite{Rubin1981}. Note that in Settings S1.\ and S2.\ we have $\overline{W}_n=1$ and thus $\widehat C_n=\widehat F_n$. This implies that condition (e) holds if ${\cal S}$ is (any subset of) the set of all sequences $(G_n)$ of distribution functions on $\R$ satisfying $G_n-F\in\bD_\phi$, $n\in\N$, and $\|G_n-F\|_\phi\to 0$; see, for instance, Theorem 2.1 in \cite{Zaehle2014}.
{\hspace*{\fill}$\Diamond$\par\bigskip}
\end{examplenorm}

\begin{examplenorm}\label{block bootstrap}
Let $(X_i)$ be a strictly stationary sequence of $\beta$-mixing random variables on $(\Omega,{\cal F},\pr)$ with distribution function $F$, and $\widehat F_n$ be given by (\ref{Def EmpDF}). Let $(\ell_n)$ be a sequence of integers such that $\ell_n\nearrow\infty$ as $n\rightarrow\infty$, and $\ell_n<n$ for all $n\in\N$. Set $k_n:=\lceil n/\ell_n\rceil$ for all $n\in\N$. Let $(I_{nj})_{n\in\N,\,1\le j\le k_n}$ be a triangular array of random variables on $(\Omega',{\cal F}',\pr')$ such that $I_{n1},\ldots,I_{nk_n}$ are i.i.d.\ according to the uniform distribution on $\{1,\ldots,n-\ell_n+1\}$ for every $n\in\N$. Define the map $\widehat F_n^*:\overline\Omega\rightarrow\bD$ by (\ref{Def F star}) with
\begin{equation}\label{example for tau - beta mixing - 2}
    W_{ni}(\omega'):=\sum_{j=1}^{k_n-1}\eins_{\{I_{nj}\le i\le I_{nj}+\ell_n-1\}}(\omega')+\eins_{\{I_{nk_n}\le i\le I_{nk_n}+(n-(k_n-1)\ell_n)-1\}}(\omega').
\end{equation}
Note that, as before, the sequence $(X_i)$ and the triangular array $(W_{ni})$ regarded as families of random variables on the product space $(\overline\Omega,\overline{\cal F},\overline\pr):=(\Omega\times\Omega',{\cal F}\otimes{\cal F}',\pr\otimes\pr')$ are independent. At an informal level this means that given a sample $X_1,\ldots,X_n$, we pick  $k_n-1$ blocks of length $\ell_n$ and one block of length $n-(k_n-1)\ell_n$ in the sample $X_1,\ldots,X_n$, where the start indices $I_{n1},I_{n2},\ldots,I_{nk_n}$ are chosen independently and uniformly in the set of indices $\{1,\ldots,n-\ell_n+1\}$:
\begin{center}
\begin{tabular}{ll}
    block $1$: \qquad & $X_{I_{n1}},X_{I_{n1}+1},\ldots,X_{I_{n1}+\ell_n-1}$\\
    block $2$: \qquad & $X_{I_{n2}},X_{I_{n2}+1},\ldots,X_{I_{n2}+\ell_n-1}$\\
    & $\vdots$\\
    block $k_{n}-1$: \qquad & $X_{I_{n(k_n-1)}},X_{I_{n(k_n-1)}+1},\ldots,X_{I_{n(k_n-1)}+\ell_n-1}$\\
    block $k_{n}$: \qquad & $X_{I_{nk_n}},X_{I_{nk_n}+1},\ldots,X_{I_{nk_n}+(n-(k_n-1)\ell_n)-1}$.
\end{tabular}
\end{center}
The bootstrapped empirical distribution function $\widehat F_n^*$ is then defined to be the distribution function of the discrete finite (not necessarily probability) measure with atoms $X_1,\ldots,X_n$ carrying masses $W_{n1},\ldots,W_{nn}$ respectively, where $W_{ni}$ specifies the number of blocks which contain $X_i$. Similar as in Lemma 5.3 in \cite{BeutnerZaehle2016} it follows that $a_n(\widehat F_n^*-\widehat{C}_n)$, with $\widehat{C}_n:=\E^{'}[\widehat F_n^*]$, takes values only in $\bD_{\phi}$ and is $(\overline{\cal F},{\cal D}_{\phi})$-measurable. That is, the first part of condition (a) of Corollary \ref{modified delta method for the bootstrap - II} holds true for $\widehat C_n:=\E^{'}[\widehat F_n^*]$. Now assume that the following assertions hold: 
\begin{itemize}
    \item[A1.] $\int\phi^p\,dF<\infty$ for some $p>4$.
    \item[A2.] The sequence of random variables $(X_i)$ is strictly stationary and $\beta$-mixing with mixing coefficients $(\beta_i)$ satisfying $\beta_i\le c\delta^i$ for some constants $c>0$ and $\delta\in(0,1)$.
    \item[A3.] The block length $\ell_n$ satisfies $\ell_n={\cal O}(n^{\gamma})$ for some $\gamma\in(0,1/2)$.
\end{itemize}
Then, as discussed in Example 4.4 and Section 5.2 of \cite{BeutnerZaehle2016}, it can be derived from a result in \cite{ArconesYu1994} that under assumptions A1.\ and A2.\ we have that condition (a) of Corollary \ref{modified delta method for non bootstrap - II} holds for $a_n:=\sqrt{n}$, $B:=B_F$, and $F_n:=F$, where $B_F$ is a centered Gaussian process with covariance function $\Gamma(t_0,t_1)=F(t_0\wedge t_1)(1-F(t_0\vee t_1))+\sum_{i=0}^1\sum_{k=2}^{\infty}\covi(\eins_{\{X_1 \le t_i\}}, \eins_{\{X_k \le t_{1-i}\}})$. Here $\bC_{\phi}$ can be chosen to be the set $\bC_{\phi,F}$ of all $v\in\bD_\phi$ whose discontinuities are also discontinuities of $F$. Moreover, Theorem \ref{corollary bracketing weight functions} below shows that under the assumptions A1.--A3.\ the second part of condition (a) (i.e.\ (\ref{modified delta method for the bootstrap - assumption - 15 - II})) and condition (e) of Corollary \ref{modified delta method for the bootstrap - II} hold for
\begin{equation}\label{def c n beta mixing case}
    \widehat C_n:=\E^{'}[\widehat F_n^*]=\frac{1}{n}\sum_{i=1}^nw_{ni}\eins_{[X_i,\infty)}\qquad\mbox{with}\qquad w_{ni}:=\ex'[W_{ni}]
\end{equation}
and the same choice of $a_n$, $B$, and $F_n$, when ${\cal S}$ is the set of all sequences $(G_n)\subseteq\bD(H)$ with $G_n-F\in\bD_\phi$, $n\in\N$, and $\|G_n-F\|_\phi\to 0$. Note that
\begin{equation}\label{values of w ni}
    w_{ni}
    =
    \left\{\begin{array}{lcl}
        k_n\,\frac{i}{n-\ell_n+1} & , & i=1,\ldots,n-(k_n-1)\ell_n\\
        (k_n-1)\,\frac{i}{n-\ell_n+1}+\frac{n-(k_n-1)\ell_n}{n-\ell_n+1} & , & i=n-(k_n-1)\ell_n+1,\ldots,\ell_n\\
        (k_n-1)\,\frac{\ell_n}{n-\ell_n+1} + \frac{n-(k_n-1)\ell_n}{n-\ell_n+1}=\frac{n}{n-\ell_n+1} & , & i=\ell_n+1,\ldots,n-\ell_n\\
        (k_n-1)\,\frac{n-i+1}{n-\ell_n+1}+\frac{2n-k_n\ell_n-i+1}{n-\ell_n+1} & , & i=n-\ell_n+1,\ldots,n-(k_n\ell_n-n)\\
        (k_n-1)\,\frac{n-i+1}{n-\ell_n+1} & , & i=n-(k_n\ell_n-n)+1,\ldots,n
    \end{array}
    \right.
\end{equation}
which can be verified easily.
{\hspace*{\fill}$\Diamond$\par\bigskip}
\end{examplenorm}

Further examples for condition (a) in Corollary \ref{modified delta method for the bootstrap - II} for dependent observations can, for example, be found in \cite{Buehlmann1994,NaikNimbalkarRajarshi1994,Peligrad1998}.

\begin{theorem}\label{corollary bracketing weight functions}
In the setting of Example \ref{block bootstrap} assume that assertions A1.--A3.\ hold, and let ${\cal S}$ be the set of all sequences $(G_n)\subseteq\bD(H)$ with $G_n-F\in\bD_\phi$, $n\in\N$, and $\|G_n-F\|_\phi\to 0$. Then the second part of assertion (a) (i.e.\ (\ref{modified delta method for the bootstrap - assumption - 15 - II})) and assertion (e) in Corollary \ref{modified delta method for the bootstrap - II} hold.
\end{theorem}

\begin{proof}
{\em Proof of second part of (a):} It is enough to show that under assumptions A1.--A3.\ the assumptions (A1)--(A4) of Theorem 1 in \cite{Buehlmann1995} hold when the class of functions is $\F_{\phi}:=\F_{\phi}^-\cup\F_{\phi}^+$, where $\F_\phi^-:=\{f_x: x\le 0\}$ and $\F_\phi^+:=\{f_x: x>0\}$ with $f_x(\cdot):=\phi(x)\eins_{(-\infty,x]}(\cdot)$ for $x \leq 0$ and $f_x(\cdot):=-\phi(x)\eins_{(x,\infty)}(\cdot)$ for $x>0$. Due to A2.\ and A3.\ we only have to verify assumptions (A3) and (A4) of Theorem 1 in \cite{Buehlmann1995}. That is, we will show that the following two assertions hold.
\begin{itemize}
    \item[1)] There exist constants $b,c>0$ such that $N_{[~]}(\varepsilon,\F_\phi,\|\cdot\|_p)\le c\,\varepsilon^{-b}$ for all $\varepsilon>0$.
    \item[2)] $\int \overline{f}^{\,p} dF<\infty$ for the envelope function $\overline{f}(z):=\sup_{x\in\R}|f_x(z)|$.
\end{itemize}
Here the bracketing number $N_{[~]}(\varepsilon,\F_\phi,\|\cdot\|_p)$ is the minimal number of $\varepsilon$-brackets w.r.t.\ $\|\cdot\|_{p}$ ($L^p$-norm w.r.t.\ $dF$) to cover $\F_\phi$, where an $\varepsilon$-bracket w.r.t.\ $\|\cdot\|_{p}$ is the set, $[\ell,u]$,  of all functions $f$ with $\ell\le f\le u$ for some Borel measurable functions $\ell,u:\R\rightarrow\R_+$ with $\ell\le u$ pointwise and $\|u-\ell\|_p\le\varepsilon$.

1): We will only show that 1) with $\F_\phi$ replaced by $\F_\phi^-$ holds true. Analogously one can show that the same holds true for $\F_\phi^+$ (and therefore for $\F_\phi$). On the one hand, since $I_p^-:=\int_{(-\infty,0]} \phi^p\,dF<\infty$ by assumption (a), we can find for every $\varepsilon>0$ a finite partition $-\infty=y_0^\varepsilon<y_1^\varepsilon<\cdots<{y_{k_\varepsilon}^\varepsilon}=0$ such that
\begin{equation}\label{block bootstrap - proof - 10}
    \max_{i=1,\ldots,k_\varepsilon}\,\int_{(y_{i-1}^\varepsilon,y_i^\varepsilon]} \phi^p\,dF \le (\varepsilon/2)^p
\end{equation}
and $k_\varepsilon\le \lceil I_p^-/(\varepsilon/2)^p\rceil$. On the other hand, using integration by parts we obtain
$$
    \int_{(-\infty,0]} F\,d(-\phi^p)\,=\,\phi(0)F(0)-\int_{(-\infty,0]}(-\phi^p)\,dF\,=\,\phi(0)F(0)+I_p^-,
$$
so that we can find a finite partition $-\infty=z_0^\varepsilon<z_1^\varepsilon<\cdots<{z_{m_\varepsilon}^\varepsilon}=0$ such that
\begin{equation}\label{block bootstrap - proof - 20}
    \max_{i=1,\ldots,m_\varepsilon}\,\int_{(z_{i-1}^\varepsilon,z_i^\varepsilon]} F\,d(-\phi^p)\le (\varepsilon/2)^p
\end{equation}
and $m_\varepsilon\le \lceil (\phi(0)F(0)+I_p^-)/(\varepsilon/2)^p\rceil$.

Now let $-\infty=x_{0}^\varepsilon<x_{1}^\varepsilon<\cdots<x_{k_{\varepsilon}+m_{\varepsilon}}^\varepsilon=0$ be the partition consisting of all points $y_i^\varepsilon$ and $z_i^\varepsilon$, and set
\begin{eqnarray}
    \ell_{i}^\varepsilon(\cdot) & := & \phi(x_{i}^\varepsilon)\eins_{(-\infty,x_{i-1}^\varepsilon]}(\cdot),\nonumber\\
    u_{i}^\varepsilon(\cdot) & := & \phi(x_{i-1}^\varepsilon)\eins_{(-\infty,x_{i-1}^\varepsilon]}(\cdot)+\phi(\cdot)\eins_{(x_{i-1}^\varepsilon,x_{i}^\varepsilon]}(\cdot).\label{def epsilon bracket}
\end{eqnarray}
Then $\ell_{i}^\varepsilon\le u_{i}^\varepsilon$. Moreover
\begin{eqnarray*}
    \|u_{i}^\varepsilon-\ell_{i}^\varepsilon\|_p
    & = & \Big(\int \big(u_{i}^\varepsilon-\ell_{i}^\varepsilon\big)^p\,dF\Big)^{1/p}\\
    & \le & \Big(\int_{(-\infty,x_{i-1}^{\varepsilon}]} \big(\phi(x_{i-1}^\varepsilon)-\phi(x_{i}^\varepsilon)\big)^p\,dF\Big)^{1/p}+\Big(\int_{(x_{i-1}^\varepsilon,x_{i}^\varepsilon]}\phi^p\,dF\Big)^{1/p}\\
    & \le & \Big(\int_{(-\infty,x_{i-1}^{\varepsilon}]} \big(\phi(x_{i-1}^\varepsilon)^p-\phi(x_{i}^\varepsilon)^p\big)\,dF\Big)^{1/p}+\varepsilon/2\\
    & \le & \Big(\big(\phi(x_{i-1}^\varepsilon)^p-\phi(x_{i}^\varepsilon)^p\big)F(x_{i-1}^\varepsilon)\Big)^{1/p}+\varepsilon/2
\end{eqnarray*}
where we used Minkovski's inequality and (\ref{block bootstrap - proof - 10}), and that $\phi$ is non-increasing on $(-\infty,0]$ and $x_{i-1}^\varepsilon\le x_{i}^\varepsilon$. Since $F$ is at least $F(x_{i-1}^\varepsilon)$ on $(x_{i-1}^\varepsilon,x_{i}^\varepsilon]$, we have
$$
    \big(\phi(x_{i-1}^\varepsilon)^p-\phi(x_{i}^\varepsilon)^p\big)F(x_{i-1}^\varepsilon)\le\int_{(x_{i-1}^\varepsilon,x_{i}^\varepsilon]}F\,d(-\phi^p)\le(\varepsilon/2)^p
$$
due to (\ref{block bootstrap - proof - 20}). Thus $\|u_{i}^\varepsilon-\ell_{i}^\varepsilon\|_p\le\varepsilon$, so that $[\ell_{i}^\varepsilon,u_{i}^\varepsilon]$ provides an $\varepsilon$-bracket w.r.t.\ $\|\cdot\|_p$. It is moreover obvious that the $\varepsilon$-brackets $[\ell_{i}^\varepsilon,u_{i}^\varepsilon]$, $i=1,\ldots,k_{\varepsilon}+m_{\varepsilon}$, cover $\F_\phi^-$. Thus, $N_{[~]}(\varepsilon,\F_\phi^-,\|\cdot\|_p)\le c\,\varepsilon^{-p}$ for a suitable constant $c>0$ and all $\varepsilon>0$.

2): The envelope function $\overline{f}$ is given by $\overline{f}(y)=\phi(y)$ for $y\leq 0$ and by $\overline{f}(y)=\phi(y-)=\phi(y)$ (recall that $\phi$ is continuous) for $y>0$. Then under assumption (a) the integrability condition 2) holds.

{\em Proof of (e):} We have to show that $\|\widehat C_n-F\|_\phi=\sup_{x\in\R}|\widehat C_n(x)-F(x)|\phi(x)\to 0$ $\pr$-a.s. We will only show that
\begin{equation}\label{block bootstrap - proof - 30}
    \sup_{x\in(-\infty,0]}|\widehat C_n(x)-F(x)|\phi(x)\longrightarrow 0\qquad\mbox{$\pr$-a.s.},
\end{equation}
because the analogue for the positive real line can be shown in the same way. Let $\ell_{i}^\varepsilon$ and $u_{i}^\varepsilon$ be as defined in (\ref{def epsilon bracket}). By assumption A1.\ we have $\int\phi\,dF<\infty$, so that similar as above we can find a finite partition $-\infty={x_0^\varepsilon}<{x_1^\varepsilon}<\cdots<{x_{k_\varepsilon+m_\varepsilon}^\varepsilon}=0$ such that $[\ell_i^\varepsilon,u_i^\varepsilon]$, $i=1,\ldots,k_\varepsilon+m_\varepsilon$, are $\varepsilon$-brackets w.r.t.\ $\|\cdot\|_1$ ($L^1$-norm w.r.t.\ $F$) covering the class $\F_\phi:=\{f_x:x\in\R\}$ introduced above. We will proceed in two steps.

{\em Step 1}. First we will show that
\begin{equation}\label{appendix a2 - eq 1}
    \sup_{x\le 0}\big|\widehat C_n(x)-F(x)\big|\phi(x)
    \,\le\,\max_{i=1,\ldots,k_\varepsilon+m_\varepsilon}\max\Big\{\int u_i^\varepsilon\,d(\widehat C_n-F)\,;\,\int \ell_i^\varepsilon\,d(F-\widehat C_n)\Big\}+\varepsilon
\end{equation}
holds true for every $\varepsilon>0$. Since $(\widehat C_n(x)-F(x))\phi(x)= \int f_x\,d\widehat C_n-\int f_x\,dF$, for (\ref{appendix a2 - eq 1}) it suffices to show
\begin{eqnarray}
    \lefteqn{\sup_{x\le 0}\Big|\int f_x\,d\widehat C_n-\int f_x\,dF\Big|}\nonumber\\
    & \le & \max_{i=1,\ldots,k_\varepsilon+m_\varepsilon}\max\Big\{\int u_i^\varepsilon\,d(\widehat C_n-F)\,;\,\int\ell_i^\varepsilon\,d(F-\widehat C_n)\Big\}+\varepsilon.
    \label{appendix a2 - eq 1 - alt}
\end{eqnarray}
To prove (\ref{appendix a2 - eq 1 - alt}), we note that for every $x\in(-\infty,y]$ there is some $i_x\in\{1,\ldots,k_\varepsilon+m_\varepsilon\}$ such that $f_x\in[\ell_{i_x}^\varepsilon,u_{i_x}^\varepsilon]$; cf.\ Step 1. Therefore, since $[\ell_{i_x}^\varepsilon,u_{i_x}^\varepsilon]$ is an $\varepsilon$-bracket w.r.t.\ $\|\cdot\|_1$,
\begin{eqnarray*}
    \int f_x\,d\widehat C_n-\int f_x\,dF
    & \le & \int u_{i_x}^\varepsilon\,d\widehat C_n-\int f_x\,dF\\
    & = & \int u_{i_x}^\varepsilon\,d(\widehat C_n-F)+\int (u_{i_x}^\varepsilon-f_x)\,dF\\
    & \le & \int u_{i_x}^\varepsilon\,d(\widehat C_n-F)+\int (u_{i_x}^\varepsilon-\ell_{i_x}^\varepsilon)\,dF\\
    & \le & \max_{i=1,\ldots,k_\varepsilon+m_\varepsilon}\int u_i^\varepsilon\,d(\widehat C_n-F)\,+\,\varepsilon.
\end{eqnarray*}
Analogously we obtain
\begin{eqnarray*}
    \int f_x\,d\widehat C_n-\int f_x\,dF & \ge & -\Big(\max_{i=1,\ldots,k_\varepsilon+m_\varepsilon}\int \ell_i^\varepsilon\,d(F-\widehat C_n)\,+\,\varepsilon\Big).
\end{eqnarray*}
That is, (\ref{appendix a2 - eq 1}) holds true.

{\em Step 2.} Because of (\ref{appendix a2 - eq 1}), for (\ref{block bootstrap - proof - 30}) to be true it suffices to show that
\begin{equation}\label{block bootstrap - proof - 40}
    \int \ell_i^\varepsilon\,d(F-\widehat C_n)\longrightarrow 0\qquad\mbox{and}\qquad\int u_i^\varepsilon\,d(\widehat C_n-F)\longrightarrow 0\qquad\pr\mbox{-a.s.}
\end{equation}
for every $i=1,\ldots,k_\varepsilon+m_\varepsilon$. We will only show the second convergence in (\ref{block bootstrap - proof - 40}), the first convergence can be shown even easier. We have
\begin{eqnarray*}
    \int u_i^\varepsilon\,d(\widehat C_n-F)
    & = & \frac{1}{n}\sum_{j=1}^n\Big(w_{ni}\, \phi(y_{i-1}^\varepsilon)\eins_{(-\infty,y_{i-1}^\varepsilon]}(X_j)-\ex\big[\phi(y_{i-1}^\varepsilon)\eins_{(-\infty,y_{i-1}^\varepsilon]}(X_1)\big]\Big)\\
    &  & +\,\frac{1}{n}\sum_{j=1}^n\Big(w_{ni}\,\phi(X_j)\eins_{(y_{i-1}^\varepsilon,y_i^\varepsilon]}(X_j)-\ex\big[\phi(X_1)\eins_{(y_{i-1}^\varepsilon,y_i^\varepsilon]}(X_1)\big]\Big)\\
    & =: & S_1(n)+S_2(n).
\end{eqnarray*}
The first summand on the right-hand side of
\begin{eqnarray*}
    S_2(n)
    & = & \frac{1}{n}\sum_{j=1}^n\Big(\phi(X_j)\eins_{(y_{i-1}^\varepsilon,y_i^\varepsilon]}(X_j)-\ex\big[\phi(X_1)\eins_{(y_{i-1}^\varepsilon,y_i^\varepsilon]}(X_1)\big]\Big)\\
    &  & +\,\frac{1}{n}\sum_{j=1}^n(w_{ni}-1)\phi(X_j)\eins_{(y_{i-1}^\varepsilon,y_i^\varepsilon]}(X_j)
\end{eqnarray*}
converges $\pr$-a.s.\ to $0$ by Theorem 1\,(ii) (and Application 5, p.\,924) in \cite{Rio1995} and our assumption A1. The second summand converges $\pr$-a.s.\ to $0$ too, which can be seen as follows. From (\ref{values of w ni}) we obtain for $n$ sufficiently large
$$
    |w_{ni}-1|
    \le
    \left\{\begin{array}{lcl}
        2 & , & i=1,\ldots,\ell_n\\
        \frac{\ell_n-1}{n-\ell_n+1} & , & i=\ell_n+1,\ldots,n-\ell_n\\
        2 & , & i=n-\ell_n+1,\ldots,n\\
    \end{array}
    \right.,
$$
so that for $n$ sufficiently large
\begin{eqnarray*}
    \lefteqn{\Big|\frac{1}{n}\sum_{j=1}^n(w_{ni}-1)\phi(X_j)\eins_{(y_{i-1}^\varepsilon,y_i^\varepsilon]}(X_j)\Big|}\\
    & \le & \frac{\ell_n-1}{n-\ell_n+1}\,\frac{1}{n}\sum_{j=\ell_n+1}^{n-\ell_n}\phi(X_j)\eins_{(y_{i-1}^\varepsilon,y_i^\varepsilon]}(X_j)\\
    & & +\,2\,\frac{2\ell_n}{n}\, \frac{1}{2\ell_n}\Big(\sum_{j=1}^{\ell_n}\phi(X_j)\eins_{(y_{i-1}^\varepsilon,y_i^\varepsilon]}(X_j)+\sum_{j=n-\ell_n+1}^n\phi(X_j)\eins_{(y_{i-1}^\varepsilon,y_i^\varepsilon]}(X_j)\Big)\\
    & =: & S_{2,1}(n)+S_{2,2}(n).
\end{eqnarray*}
We have seen above that $\frac{1}{n}\sum_{j=1}^{n}\phi(X_j)\eins_{(y_{i-1}^\varepsilon,y_i^\varepsilon]}(X_j)$ converges $\pr$-a.s.\ to the constant $\ex[\phi(X_1)\,\eins_{(y_{i-1}^\varepsilon,y_i^\varepsilon]}(X_1)]$. Since $\ell_n$ converges to $\infty$ at a slower rate than $n$ (by assumption A3.), it follows that $S_{2,1}(n)$ converges $\pr$-a.s.\ to $0$. Using the same arguments we obtain that $S_{2,2}(n)$ converges $\pr$-a.s.\ to $0$. Hence $S_2(n)$ converges $\pr$-a.s.\ to $0$. Analogously one can show that $S_1(n)$ converges $\pr$-a.s.\ to $0$.
\end{proof}


\section{Examples for uniformly quasi-Hadamard differentiable functionals}\label{Examples}


\subsection{Average Value at Risk functional}\label{Average Value at Risk functional}

Let $(\Omega,{\cal F},\pr)$ be an atomless probability space and $L^1=L^1(\Omega,{\cal F},\pr)$ be the usual $L^1$-space. The Average Value at Risk at level $\alpha\in(0,1)$ is the map $\avatr_\alpha:L^1\rightarrow\R$ defined by
\begin{equation}\label{Def AVaR}
    \avatr_\alpha(X):=\int_\alpha^1 F_X^\leftarrow(s)\,ds=-\int_{-\infty}^0g_\alpha(F_X(x))\,dx+\int_0^\infty\big(1-g_\alpha(F_X(x))\big)\,dx,
\end{equation}
where $g_\alpha(t):=\frac{1}{1-\alpha}\max\{t-\alpha;0\}$ and $F_X^\leftarrow(s):=\inf\{x\in\R:F_X(x)\ge s\}$ denotes the left-continuous inverse of the distribution function $F_X$ of $X$. Note that $\avatr_\alpha(X)=\ex[X|X\ge F_X^\leftarrow(\alpha)]$ when $F_X$ is continuous at $F_X^\leftarrow(\alpha)$, and that $\avatr_\alpha$ is one of the most popular risk measures in practice. In view of the second identity in (\ref{Def AVaR}) we may associate with $\avatr_\alpha$ the statistical functional ${\cal R}_\alpha:\mathbf{F}_1\rightarrow\R$ defined by
\begin{equation}\label{def avarf}
    {\cal R}_\alpha(F):=-\int_{-\infty}^0g_\alpha(F(x))\,dx+\int_0^\infty\big(1-g_\alpha(F(x))\big)\,dx,
\end{equation}
where $\mathbf{F}_1$ is the set of the distribution functions $F_X$ of all $X\in L^1$. Using the notation introduced at the beginning of Section \ref{application to statistical functionals}, we obtain the following result.

\begin{proposition}\label{proposition AVaR}
Let $F\in\mathbf{F}_1$ and assume that $F$ takes the value $1-\alpha$ only once. Let ${\cal S}$ be the set of all sequences $(G_n)\subseteq\F_1$ with $G_n\to F$ pointwise. Moreover assume that $\int 1/\phi(x)\,dx<\infty$. Then the map ${\cal R}_\alpha:\mathbf{F}_1\,(\subseteq\mathbf{D})\rightarrow\R$ is uniformly quasi-Hadamard differentiable w.r.t.\ $\mathcal{S}$ tangentially to $\bD_\phi\langle\bD_\phi\rangle$, and the uniform quasi-Hadamard derivative $\dot{\cal R}_{\alpha;F}:\bD_{\phi}\rightarrow\R$ is given by
\begin{equation}\label{def of qh ableitung von T}
    \dot{\cal R}_{\alpha;F}(v):= -\int g_\alpha'(F(x))v(x)\,dx
\end{equation}
with $g_\alpha'(t):=\frac{1}{1-\alpha}\eins_{(1-\alpha,1]}(t)$.
\end{proposition}

Proposition \ref{proposition AVaR} shows in particular that for any $F\in\mathbf{F}_1$ which takes the value $1-\alpha$ only once, the map ${\cal R}_\alpha:\mathbf{F}_1\,(\subseteq\mathbf{D})\rightarrow\R$ is uniformly quasi-Hadamard differentiable at $F$ tangentially to $\bD_\phi\langle\bD_\phi\rangle$ (in the sense of part (ii) of Definition \ref{definition quasi hadamard}) with uniform quasi-Hadamard derivative given by (\ref{def of qh ableitung von T}).

\bigskip

\begin{proof}(of Proposition \ref{proposition AVaR})
First of all note that the map $\dot{\cal R}_{\alpha;F}$ defined in (\ref{def of qh ableitung von T}) is continuous w.r.t.\ $\|\cdot\|_\phi$, because
$$
    |\dot{\cal R}_{\alpha;F}(v_1)-\dot{\cal R}_{\alpha;F}(v_2)|
    \le \int\frac{1}{1-\alpha}\,|v_1(x)-v_2(x)|\,dx
    \le \Big(\frac{1}{1-\alpha}\int 1/\phi(x)\,dx\Big)\|v_1-v_2\|_\phi
$$
holds for every $v_1,v_2\in\mathbf{D}_\phi$. 

Now, let $((F_n),v,(v_n),(\varepsilon_n))$ be a quadruple with $(F_n)\subseteq\mathbf{F}_1$ satisfying $F_n\to F$ pointwise, $v\in \mathbf{D}_\phi$, $(v_n)\subseteq\mathbf{D}_\phi$ satisfying $\|v_n-v\|_\phi\to 0$ and $(F_n+\varepsilon_nv_n)\subseteq\mathbf{F}_1$, and $(\varepsilon_n)\subseteq(0,\infty)$ satisfying $\varepsilon_n\to 0$. It remains to show that
$$
    \lim_{n\to\infty}\Big|\frac{{\cal R}_\alpha(F_n+\varepsilon_nv_n)-{\cal R}_\alpha(F_n)}{\varepsilon_n}-\dot{\cal R}_{\alpha;F}(v)\Big|=\,0,
$$
that is, in other words, that
\begin{equation}\label{proof qhd of t - 20}
    \lim_{n\to\infty}\Big|\int\Big(\frac{  g_\alpha\big(F_n(x)\big)-g_\alpha\big((F_n+\varepsilon_nv_n)(x)\big) }{\varepsilon_n}-\big(-g_\alpha'(F(x))v(x)\big)\Big)\,dx\Big|=\,0.
\end{equation}
Let us denote the integrand of the integral in (\ref{proof qhd of t - 20}) by $I_n(x)$. In virtue of $F_n\to F$ pointwise, $\|v_n-v\|_\phi\to 0$, $\varepsilon_n\to 0$, and
$$
    |(F_n+\varepsilon_nv_n)(x)-F(x)|\,\le\,|F_n(x)-F(x)|+\varepsilon_n|v_n(x)-v(x)|+\varepsilon_n|v(x)|,
$$
we have $\lim_{n\to\infty}F_n(x)=F(x)$ and $\lim_{n\to\infty}(F_n(x)+\varepsilon_nv_n(x))=F(x)$ for every $x\in\R$. Thus, for every $x\in\R$ with $F(x)<1-\alpha$ we obtain $g_\alpha'(F(x))v(x)=0$ and
$$
    \frac{g_\alpha\big(F_n(x)\big)-g_\alpha\big((F_n+\varepsilon_nv_n)(x)\big)}{\varepsilon_n}\,=\,0\qquad\mbox{for sufficiently large }n,
$$
i.e.\ $\lim_{n\to\infty}I_n(x)=0$. Moreover for every $x\in\R$ with $F(x)>1-\alpha$ we obtain $g_\alpha'(F(x))v(x)=\frac{1}{1-\alpha}v(x)$ and
$$
    \frac{g_\alpha\big(F_n(x)\big)-g_\alpha\big((F_n+\varepsilon_nv_n)(x)\big)}{\varepsilon_n}=-\frac{v_n(x)}{1-\alpha}\qquad\mbox{for sufficiently large }n,
    $$
i.e.\ $\lim_{n\to\infty}I_n(x)=0$. Since we assumed that $F$ takes the value $1-\alpha$ only once, we can conclude that $\lim_{n\to\infty}I_n(x)=0$ for Lebesgue-a.e.\ $x\in\R$. Moreover, by the Lipschitz continuity of $g_\alpha$ with Lipschitz constant $\frac{1}{1-\alpha}$ we have
\begin{eqnarray*}
    |I_n(x)|
    & = & |I_n(x)|\,\phi(x)\,\phi(x)^{-1}\\
    & = & \Big|\frac{g_\alpha\big(F_n(x)\big)-g_\alpha\big((F_n+\varepsilon_nv_n)(x)\big)}{\varepsilon_n}+g_\alpha'(F(x))v(x)\Big|\,\phi(x)\,\phi(x)^{-1}\\
    & \le & \frac{1}{1-\alpha}\big(\|v_n\|_\phi+\|v\|_\phi\big)\,\phi(x)^{-1}\\
    & \le & \frac{1}{1-\alpha}\big(\sup_{n\in\N}\|v_n\|_\phi+\|v\|_\phi\big)\,\phi(x)^{-1}.
\end{eqnarray*}
Since $\sup_{n\in\N}\|v_n\|_\phi<\infty$ (recall $\|v_n-v\|_\phi\to 0$), the assumption $\int1/\phi(x)\,dx<\infty$ ensures that the latter expression provides a Borel measurable majorant of $I_n$. Now, the Dominated Convergence theorem implies (\ref{proof qhd of t - 20}).
\end{proof}

As an immediate consequence of Theorem \ref{modified delta method for the bootstrap - II}, Examples \ref{iid bootstrap} and \ref{block bootstrap}, and Proposition \ref{proposition AVaR} we obtain the following corollary.

\begin{corollary}
Let $F$, $\widehat F_n$, $\widehat F_n^*$, $\widehat C_n$, and $B_F$ be as in Example \ref{iid bootstrap} (S1.\ or S2.) or as in Example \ref{block bootstrap} respectively, and assume that the assumptions discussed in Example \ref{iid bootstrap} or in Example \ref{block bootstrap} respectively are fulfilled for some weight function $\phi$ with $\int 1/\phi(x)\,dx<\infty$ (in particular $F\in\F_1$). Then
$$
    \sqrt{n}\big({\cal R}_\alpha(\widehat F_n)-{\cal R}_\alpha(F)\big)\,\leadsto\,\dot{\cal R}_{\alpha;F}(B_F)\qquad\mbox{in $(\R,{\cal B}(\R))$}
$$
and
$$
    \sqrt{n}\big({\cal R}_\alpha(\widehat F_n^*(\omega,\cdot))-{\cal R}_\alpha(\widehat C_n(\omega))\big)\,\leadsto\,\dot {\cal R}_{\alpha;F}(B_F)\qquad\mbox{in $(\R,{\cal B}(\R))$},\qquad\mbox{$\pr$-a.e.\ $\omega$}.
$$
\end{corollary}

For the bootstrap scheme S1.\ in Example \ref{iid bootstrap} the result of the preceding corollary can be also deduced from Theorem 7 in \cite{Gribkova2002}. According to \cite{Gribkova2016}, condition (1) of this theorem is satisfied if there are $0=a_0<a_1<\cdots<a_k=1$ for some $k\in\N$ such that $J$ is Hölder continuous on each interval $(a_{i-1},a_i)$, $1\le i\le k$, and the measure $dF^{-1}$ has no mass at the points $a_1,\ldots,a_{k-1}$.


\subsection{Compound distribution functional}\label{Convolution functional}

Let $p=(p_k)_{k \in \N_0}$ be a sequence in $\R_+$ with $\sum_{k=0}^{\infty} p_k=1$, so that $p$ specifies the distribution of a count variable $N$. Let $\mathbf{F}$ denote the set of all distribution functions on $\R$, and consider the functional ${\cal C}_{p}:\mathbf{F}\rightarrow\mathbf{F}$ defined by
\begin{equation}\label{def codifu}
    {\cal C}_{p}(F):=\sum_{k=0}^{\infty} p_k F^{*k},
\end{equation}
where $F^{*k}$ refers to the $k$-fold convolution of $F$, that is, $F^{*0}:=\eins_{[0,\infty)}$ and
\begin{eqnarray*}
    F^{*k}(x)
    & := & \int F(x-x_{k-1})\,dF^{*(k-1)}(x_{k-1})\\
    & = & \int\cdots\int F(x-x_{k-1}-\cdots-x_1)\,dF(x_1)\cdots dF(x_{k-1})
\end{eqnarray*}
for $k\in\N$. If $p_m=1$ for some $m\in\N_0$, then ${\cal C}_p(F)=F^{*m}$.

For any $\lambda \geq 0$, let the function $\phi_{\lambda}: \R \rightarrow [1,\infty)$ be defined by $\phi_{\lambda}(x):=(1+|x|)^{\lambda}$ and denote by $\mathbf{F}_{\phi_{\lambda}}$ the set of all distribution functions $F$ that satisfy $\int \phi_{\lambda}(x)\,dF(x)<\infty$. Using the notation introduced at the beginning of Section \ref{application to statistical functionals} and the terminology of part (ii) of Definition \ref{definition quasi hadamard}, we obtain the following Proposition \ref{QHD of CF}. In the proposition the functional ${\cal C}_p$ is restricted to the domain $\F_{\phi_\lambda}$ in order to obtain $\bD_{\phi_{\lambda'}}$ as the corresponding trace. The latter will be important for Corollary \ref{QHD of Composition}.

\begin{proposition}\label{QHD of CF}
Let $\lambda > \lambda' \geq 0$ and $F\in\F_{\phi_{\lambda}}$. Assume that $\sum_{k=1}^{\infty} p_k\,k^{(1+\lambda)\vee 2} < \infty$. Then the map ${\cal C}_{p}: \mathbf{F}_{\phi_{\lambda}} (\subseteq \mathbf{D})  \rightarrow \mathbf{F}(\subseteq \mathbf{D})$ is uniformly quasi-Hadamard differentiable at $F$ tangentially to $\bD_{\phi_{\lambda}}\langle\bD_{\phi_{\lambda}}\rangle$ with trace $\bD_{\phi_{\lambda'}}$. Moreover, the uniform quasi-Hadamard derivative $\dot{\cal C}_{p;F}:\bD_{\phi_{\lambda}}\rightarrow\bD_{\phi_{\lambda'}}$ is given by
\begin{equation}\label{def of qh ableitung von C}
    \dot{\cal C}_{p;F}(v)(\cdot):= v * H_{p,F}(\,\cdot\,):=\int v(\,\cdot\,-x)\,dH_{p,F}(x),
\end{equation}
where $H_{p,F}:=\sum_{k=1}^{\infty}k\,p_kF^{*(k-1)}$. In particular, if $p_m=1$ for some $m\in\N$, then
$$
    \dot{\cal C}_{p;F}(v)(\cdot)=m\int v(\,\cdot\,-x)\,dF^{*(m-1)}(x).
$$
\end{proposition}

Proposition \ref{QHD of CF} extends Proposition 4.1 of \cite{Pitts1994}. Before we prove the proposition, we note that the proposition together with Theorem \ref{modified delta method for the bootstrap - II} and Examples \ref{iid bootstrap} and \ref{block bootstrap} yields the following corollary.

\begin{corollary}
Let $F$, $\widehat F_n$, $\widehat F_n^*$, $\widehat C_n$, and $B_F$ be as in Example \ref{iid bootstrap} (S1.\ or S2.) or as in Example \ref{block bootstrap} respectively, and assume that the assumptions discussed in Example \ref{iid bootstrap} or in Example \ref{block bootstrap} respectively are fulfilled for some weight function $\phi$ with $\int 1/\phi(x)\,dx<\infty$ (in particular $F\in\F_1$). Then for $\lambda'\in(0,\lambda)$
$$
    \sqrt{n}\big({\cal C}_p(\widehat F_n)-{\cal C}_p(F)\big)\,\leadsto^\circ\,\dot{\cal C}_{p;F}(B_F)\qquad\mbox{in $(\bD_{\phi_\lambda'},{\cal D}_{\phi_\lambda'},\|\cdot\|_{\phi_{\lambda'}})$}
$$
and
$$
    \sqrt{n}\big({\cal C}_p(\widehat F_n^*(\omega,\cdot))-{\cal C}_p(\widehat C_n(\omega))\big)\,\leadsto^\circ\,\dot {\cal C}_{p;F}(B_F)\qquad \mbox{in $(\bD_{\phi_\lambda'},{\cal D}_{\phi_\lambda'},\|\cdot\|_{\phi_{\lambda'}})$},\qquad\mbox{$\pr$-a.e.\ $\omega$}.
$$
\end{corollary}

To ease the exposition of the proof of Proposition \ref{QHD of CF} we first state a lemma that follows from results given in \cite{Pitts1994}. In the sequel we will use $f*H$ to denote the function defined by $f*H(\cdot):=\int v(\,\cdot\,-x)\,dH(x)$ for any measurable function $f$ and any distribution function $H$ of a finite (not necessarily probability) Borel measure on $\R$ for which $f*H(\cdot)$ is well defined on $\R$.

\begin{lemma}\label{lemma preceding qHD of compound}
Let $\lambda>\lambda' \geq 0$, and $(F_n)\subseteq\mathbf{F}_{\phi_{\lambda}}$ and $(G_n) \subseteq \mathbf{F}_{\phi_{\lambda}}$ be any sequences such that $\|F_n-F\|_{\phi_\lambda}\to 0$ and $\|G_n-G\|_{\phi_\lambda}\to 0$ for some $F,G\in\F_{\phi_{\lambda}}$. Then the following two assertions hold.
\begin{itemize}
    \item[(i)] There exists a constant $C_1>0$ such that for every $k, n \in \N$
    $$
        \|\eins_{[0,\infty)}-F_n^{*k}\|_{\phi_{\lambda'}} \leq (2^{\lambda'-1} \vee 1)(1 + k^{\lambda' \vee 1} C_1).
    $$
    \item[(ii)] For every $v \in \bD_{\phi_{\lambda'}}$ there exists  a constant $C_2>0$ such that for every $k, \ell, n \in \N$
    $$
        \| v * (F_n^{*k}*G_n^{*\ell})\|_{\phi_{\lambda'}} \le 2^{\lambda'}\big(1+2^{\lambda'}(2^{\lambda'-1} \vee 1)(2+(k+\ell)^{\lambda' \vee 1}C_2)\big)\|v\|_{\phi_{\lambda'}}.
    $$
\end{itemize}
\end{lemma}

\begin{proof}
(i): From (2.4) in \cite{Pitts1994} we have
\begin{eqnarray*}
    \|\eins_{[0,\infty)}-F_n^{*k}\|_{\phi_{\lambda'}} \leq (2^{\lambda'-1} \vee 1)\Big(1 + k^{\lambda' \vee 1} \int |x|^{\lambda'} \,dF_n(x)\Big),
\end{eqnarray*}
so that it remains to show that $\int |x|^{\lambda'} \,dF_n(x)$ is bounded above uniformly in $n\in\N$. The functions $\eins_{[0,\infty)}-F_n$ and $\eins_{[0,\infty)}-F$ lie in $\bD_{\phi_\lambda}$, because $F_n,F\in\F_{\phi_\lambda}$. Along with $\|F_n-F\|_{\phi_{\lambda}} \to 0$ this implies $\int |x|^{\lambda'}\,dF_n(x) \rightarrow \int |x|^{\lambda'}\,dF(x)$; see Lemma 2.1 in \cite{Pitts1994}. 
Therefore, $\int |x|^{\lambda'}\,dF_n(x) \leq C_1$ for some suitable finite constant $C_1>0$ and all $n\in\N$.

(ii): With the help of Lemma 2.3 of \cite{Pitts1994} (along with $\|F_n^{*k}*G_n^{*\ell}\|_{\infty}=1$), Lemma 2.4 of \cite{Pitts1994}, and Equation (2.4) in \cite{Pitts1994} we obtain
\begin{eqnarray*}
    \lefteqn{\|v * (F_n^{*k}*G_n^{*\ell})\|_{\phi_{\lambda'}}}\\
    & \le & 2^{\lambda'} \|v\|_{\phi_{\lambda'}}\big(1+\|\eins_{[0,\infty)}-F_n^{*k}*G_n^{*\ell}\|_{\phi_{\lambda'}}\big)\\
    & \le & 2^{\lambda'} \|v\|_{\phi_{\lambda'}} \big(1+2^{\lambda'}\big(\|\eins_{[0,\infty)}-F_n^{*k}\|_{\phi_{\lambda'}}+\|\eins_{(0,\infty)-}G_n^{*\ell}\|_{\phi_{\lambda'}}\big)\big)\\
    & \leq & 2^{\lambda'} \|v\|_{\phi_{\lambda'}} \Big(1+2^{\lambda'}(2^{\lambda'-1} \vee 1)\Big(1+k^{\lambda'\vee 1} \int |x|^{\lambda'}\,dF_n(x) + 1 + \ell^{\lambda'\vee 1 } \int |x|^{\lambda'}\,dG_n(x) \Big)\Big).
\end{eqnarray*}
So it remains to show that $\int |x|^{\lambda'} \,dF_n(x)$ and $\int |x|^{\lambda'} \,dG_n(x)$ are bounded above uniformly in $n\in\N$. But this was already done in the proof of part (i).
\end{proof}

\bigskip

\begin{proof}{\bf of Proposition \ref{QHD of CF}}
First, note that for $G_1,G_2\in\mathbf{F}_{\phi_{\lambda}}$ we have
\begin{eqnarray*}
    \|{\cal C}_{p}(G_1)-{\cal C}_{p}(G_2)\|_{\phi_{\lambda'}}
    & \le &  \|{\cal C}_{p}(G_1)-\eins_{[0,\infty)}\|_{\phi_{\lambda'}}+\|I_{[0,\infty)}- {\cal C}_{p}(G_")\|_{\phi_{\lambda'}}\\
    & \leq & \int (1+|x|)^{\lambda'}\,d\mathcal{C}_p(G_1)(x) + \int (1+|x|)^{\lambda'}\,d\mathcal{C}_p(G_2)(x)
\end{eqnarray*}
by Equation (2.1) in \cite{Pitts1994}. Moreover, according to Lemma 2.2 in \cite{Pitts1994} we have that the integrals $\int |x|^{\lambda'}d\mathcal{C}_p(F)(x)$ and $\int |x|^{\lambda'}d\mathcal{C}_p(G)(x)$ are finite under the assumptions of the proposition. Hence, $\bD_{\phi_{\lambda'}}$ can indeed be seen as the trace.

Second, we show $(\|\cdot\|_{\phi_\lambda},\|\cdot\|_{\phi_{\lambda'}})$-continuity of the map $\dot{\cal C}_{p;F}:\bD_{\phi_\lambda}\rightarrow\bD_{\phi_{\lambda'}}$. To this end let $v\in \mathbf{D}_{\phi_{\lambda}}$ and $(v_n)\subseteq\mathbf{D}_{\phi_{\lambda}}$ such that $\|v_n-v\|_{\phi_{\lambda}}\to 0$. For every $k \in \N$ we have
\begin{eqnarray*}
    \lefteqn{\|p_k k(v_n-v)*F^{*(k-1)}\|_{\phi_{\lambda'}}}\\
    & \le & 2^{\lambda'} \|v_n-v\|_{\phi_{\lambda'}}\,p_k\,k\big(\|\eins_{[0,\infty)}\,\|F^{*(k-1)}\|_{\infty}-F^{*(k-1)}\|_{\phi_{\lambda'}}+\|F^{*(k-1)}\|_{\infty}\big)\\
    & = & 2^{\lambda'} \|v_n-v\|_{\phi_{\lambda'}}\,p_k\,k\big(\|\eins_{[0,\infty)}-F^{*(k-1)}\|_{\phi_{\lambda'}}+1\big)\\
    & \le & 2^{\lambda'} \|v_n-v\|_{\phi_{\lambda'}}\,p_k\,k\Big((2^{\lambda'-1} \vee 1)\Big(1+(k-1)^{\lambda' \vee 1}\int|x|^{\lambda'}\,dF(x)\Big)+1\Big),
\end{eqnarray*}
where the first and the second inequality follow from Lemma 2.3 and Equation (2.4) in \cite{Pitts1994} respectively. Hence,
\begin{eqnarray*}
    \lefteqn{\|\dot{\cal C}_{p;F}(v_n)-\dot{\cal C}_{p;F}(v)\|_{\phi_{\lambda'}}~=~\|v_n*H_{p,F}-v*H_{p,F}\|_{\phi_{\lambda'}}}\\
    & \le & 2^{\lambda'} \|v_n-v\|_{\phi_{\lambda'}}\sum_{k=1}^{\infty} p_k\,k\Big((2^{\lambda'-1} \vee 1)\Big(1+(k-1)^{\lambda' \vee 1}\int|x|^{\lambda'}\,dF(x)\Big)+1\Big).
\end{eqnarray*}
Now, the series converges due to the assumptions, and $\|v_n-v\|_{\phi_{\lambda}} \rightarrow 0$ implies $\|v_n-v\|_{\phi_{\lambda'}} \rightarrow 0$. Thus $\|\dot{\cal C}_{p;F}(v_n)-\dot{\cal C}_{p;F}(v)\|_{\phi_{\lambda'}}\to 0$, which proves continuity.

Third, let $((F_n),v,(v_n),(\varepsilon_n))$ be a quadruple with $(F_n)\subseteq\mathbf{F}_{\phi_{\lambda}}$ satisfying $\|F_n-F\|_{\phi_{\lambda}} \to 0$, $v\in \mathbf{D}_{\phi_{\lambda}}$, $(v_n)\subseteq\mathbf{D}_{\phi_{\lambda}}$ satisfying $\|v_n-v\|_{\phi_{\lambda}}\to 0$ and $(F_n+\varepsilon_nv_n)\subseteq\mathbf{F}_{\phi_{\lambda}}$, and $(\varepsilon_n)\subseteq(0,\infty)$ satisfying $\varepsilon_n\to 0$. It remains to show that
$$
    \lim_{n\to\infty}\,\Big\|\frac{{\cal C}_p(F_n+\varepsilon_nv_n)-{\cal C}_p(F_n)}{\varepsilon_n}-\dot{\cal C}_{p;F}(v)\Big\|_{\phi_{\lambda'}}=\,0.
$$
To do so, define for $k\in\N_0$ a map $H_k: \mathbf{F} \times \mathbf{F}: \rightarrow \mathbf{F}$ by
$$
    H_k(G_1,G_2):=\sum_{j=0}^{k-1} G_1^{*(k-1-j)} * G_2^{*j}.
$$
with the usual convention that the sum over the empty sum equals zero. We find that for every $M\in\N$
\begin{eqnarray}
    \lefteqn{\Big\|\frac{{\cal C}_p(F_n+\varepsilon_nv_n)-{\cal C}_p(F_n)}{\varepsilon_n}-\dot{\cal C}_{p;F}(v)\Big\|_{\phi_{\lambda'}}} \nonumber \\
    & = & \Big\|\frac{1}{\varepsilon_n}\Big(\sum_{k=0}^{\infty} p_k(F_n+\varepsilon_n v_n)^{*k} - \sum_{k=0}^{\infty} p_k F_n^{*k} \Big) - \dot{\cal C}_{p;F}(v)\Big\|_{\phi_{\lambda'}} \nonumber  \\
    & = & \Big\|\frac{1}{\varepsilon_n}\Big(\sum_{k=1}^{\infty} \big(p_k(F_n+\varepsilon_n v_n)^{*k} - p_kF_n^{*k}\big) \Big) - \dot{\cal C}_{p;F}(v)\Big\|_{\phi_{\lambda'}} \nonumber \\
    & = & \Big\|\sum_{k=1}^{\infty} p_k v_n * H_k(F_n+\varepsilon_n v_n,F_n) - \dot{\cal C}_{p;F}(v)\Big\|_{\phi_{\lambda'}} \nonumber \\
    & \le & \Big\|\sum_{k=M+1}^{\infty} p_k v_n * H_k(F_n+\varepsilon_n v_n,F_n)\Big\|_{\phi_{\lambda'}}+  \Big\| \sum_{k=1}^M p_k(v_n-v)*H_k(F_n+\varepsilon_n v_n,F_n)\Big\|_{\phi_{\lambda'}}  \nonumber \\
    & & +\Big\| v*\sum_{k=M+1}^{\infty} kp_kF^{*(k-1)}\Big\|_{\phi_{\lambda'}} + \Big\| \sum_{k=1}^M p_k v*H_k(F_n+\varepsilon_n v_n,F_n)-kp_kv*F^{*(k-1)}\Big\|_{\phi_{\lambda'}}\nonumber\\
    & =: & S_1(n,M)+S_2(n,M)+S_3(M)+S_4(n,M),\nonumber
\end{eqnarray}
where for the third ``$=$'' we use the fact that for $G_1,G_2\in\mathbf{F}$
\begin{equation}\label{telescoping sum convolution}
    (G_1-G_2)*H_k(G_1,G_2)= G_1^{*k} - G_2^{*k}.
\end{equation}
By part (ii) of Lemma \ref{lemma preceding qHD of compound} (this lemma can be applied since $\|F_n+\varepsilon_n v_n -F\|_{\phi_{\lambda}} \rightarrow 0$) there exists a constant $C_2>0$ such that for all $n\in\N$ 
\begin{eqnarray}\label{proof convolution quasi 1}
    S_1(n,M)
    & = & \Big\|\sum_{k=M+1}^{\infty} p_k v_n * H_k(F_n+\varepsilon_n v_n,F_n)\Big\|_{\phi_{\lambda'}} \nonumber \\
    & \le & 2^{\lambda'}\|v_n\|_{\phi_{\lambda'}} \sum_{k=M+1}^{\infty} p_k\, k \big(1+2^{\lambda'}(2^{\lambda'-1} \vee 1)\big(2+(k-1)^{\lambda' \vee 1}C_2\big)\big).
\end{eqnarray}
Since $\lambda'<\lambda$ and $\|v_n-v\|_{\phi_{\lambda}} \rightarrow 0$, we have $\|v_n\|_{\phi_{\lambda'}} \leq K_1$ for some finite constant $K_1>0$ and all $n\in\N$. Hence, the right-hand side of (\ref{proof convolution quasi 1}) can be made arbitrarily small by choosing $M$ large enough. That is, $S_1(n,M)$ can be made arbitrarily small uniformly in $n\in\N$ by choosing $M$ large enough.

Furthermore, it is demonstrated in the proof of Proposition 4.1 of \cite{Pitts1994} that $S_3(M)$ can be made arbitrarily small by choosing $M$ large enough.

Next, applying again part (ii) of Lemma \ref{lemma preceding qHD of compound} 
we obtain
\begin{eqnarray*}
    S_2(n,M)
    & = & \Big\| \sum_{k=1}^M p_k(v_n-v)*H_k(F_n+\varepsilon_n v_n,F_n)\Big\|_{\phi_{\lambda'}} \\
    & \le & 2^{\lambda'} \sum_{k=1}^M p_k \,k\,\|v_n-v\|_{\phi_{\lambda'}}\big(1+2^{\lambda'}(2^{\lambda'-1} \vee 1)\big(2+(k-1)^{\lambda' \vee 1}C_2\big)\big).
\end{eqnarray*}
Using $\|v_n-v\|_{\phi_{\lambda'}} \leq \|v_n-v\|_{\phi_\lambda} \rightarrow 0$ this term tends to zero as $n \rightarrow \infty$ for a given $M$.

It remains to consider the summand
\begin{eqnarray*}
    S_4(n,M)
    & = & \Big\| \sum_{k=1}^M p_k v*H_k(F_n+\varepsilon_n v_n,F_n)-kp_kv*F^{*(k-1)}\Big\|_{\phi_{\lambda'}}\\
    & = &  \Big\| \sum_{k=1}^M p_k\sum_{\ell=0}^{k-1} \Big(v*(F_n+\varepsilon_nv_n)^{*(k-1-\ell)}*F_n^{*\ell}-v*F^{*(k-1)}\Big)\Big\|_{\phi_{\lambda'}}.
\end{eqnarray*}
We will show that for $M$ fixed this term can be made arbitrarily small by letting $n \rightarrow \infty$. This would follow if for every given $k\in\{1,\ldots,M\}$ and $\ell\in\{0,\ldots,k-1\}$ the expression
$$
    \|v*(F_n+\varepsilon_nv_n)^{*(k-1-\ell)}*F_n^{*\ell}-v*F^{*(k-1)}\|_{\phi_{\lambda'}}
$$
could be made arbitrarily small by letting $n \rightarrow \infty$. For every such $k$ and $\ell$ we can find a linear combination of indicator functions of the form $\eins_{[a,b)}$, $-\infty<a<b<\infty$, which we denote by $\widetilde v$, such that
\begin{eqnarray}\label{eq proof convoultion triangle three times}
    \lefteqn{\|v*(F_n+\varepsilon_nv_n)^{*(k-1-\ell)}*F_n^{*\ell}-v*F^{*(k-1)}\|_{\phi_{\lambda'}}} \nonumber \\
    & \le & \|v*(F_n+\varepsilon_nv_n)^{*(k-1-\ell)}*F_n^{*\ell}-\widetilde{v}*(F_n+\varepsilon_nv_n)^{*(k-1-\ell)}*F_n^{*\ell}\|_{\phi_{\lambda'}} \nonumber \\\nonumber
    & & +\,\|\widetilde{v}*(F_n+\varepsilon_nv_n)^{*(k-1-\ell)}*F_n^{*\ell}-\widetilde{v}*F^{*(k-1)}\|_{\phi_{\lambda'}} \nonumber \\
    & & +\,\|\widetilde{v}*F^{*(k-1)} - v*F^{*(k-1)}\|_{\phi_{\lambda'}} \nonumber \\
    & \le & 2^{\lambda'}\|\widetilde{v}-v\|_{\phi_{\lambda'}} \big(\|\eins_{[0,\infty)} - (F_n+\varepsilon_nv_n)^{*(k-1-\ell)}*F_n^{*\ell}\|_{\phi_{\lambda'}}+1\big) \nonumber \\
    & & +\,c(\lambda',\widetilde{v})\,\|(F_n+\varepsilon_nv_n)^{*(k-1-\ell)}*F_n^{*\ell}-F^{*(k-1)}\|_{\phi_{\lambda'}} \nonumber \\
    & & +\,2^{\lambda'}\,\|\widetilde{v}-v\|_{\phi_{\lambda'}}\big(\|\eins_{[0,\infty)}- F^{*(k-1)}\|_{\phi_{\lambda'}}+1\big)
\end{eqnarray}
for some suitable finite constant $c(\lambda',\widetilde{v})>0$ depending only on $\lambda'$ and $\widetilde v$. The first inequality in (\ref{eq proof convoultion triangle three times}) is obvious (and holds for any $\widetilde v\in\bD_{\phi_{\lambda'}}$). The second inequality in (\ref{eq proof convoultion triangle three times}) is obtained by applying Lemma 2.3 of \cite{Pitts1994} to the first summand (noting that $\|(F_n+\varepsilon_nv_n)^{*(k-1-\ell)}*F_n^{*\ell}\|_{\infty}=1$; recall $F_n+\varepsilon_nv_n\in \mathbf{F}$), by applying Lemma 4.3 of \cite{Pitts1994} to the second summand (which requires that $\widetilde v$ is as described above), and by applying Lemma 2.3 of \cite{Pitts1994} to the third summand.

We now consider the three summands on the right-hand side of (\ref{eq proof convoultion triangle three times}) separately. We start with the third term. Since $v \in \bD_{\phi_{\lambda}}$, Lemma 4.2 of \cite{Pitts1994} ensures that we may assume that $\widetilde{v}$ is chosen such that $\|\widetilde{v}-v\|_{\phi_{\lambda'}}$ is arbitrarily small. Hence, for fixed $M$ the third summand in (\ref{eq proof convoultion triangle three times}) can be made arbitrarily small.

We next consider the the second summand in (\ref{eq proof convoultion triangle three times}). Obviously, 
\begin{eqnarray}\label{eq proof convolution telesum last}
    \lefteqn{\|(F_n+\varepsilon_nv_n)^{*(k-1-\ell)}*F_n^{*\ell}-F^{*(k-1)}\|_{\phi_{\lambda'}}}\nonumber\\
    & = & \|(F_n+\varepsilon_nv_n)^{*(k-1-\ell)}*F_n^{*\ell}-F_n^{*(k-1)}+F_n^{*(k-1)}-F^{*(k-1)}\|_{\phi_{\lambda'}} \nonumber \\
    & \le & \big\|\big((F_n+\varepsilon_nv_n)^{*(k-1-\ell)}-F_n^{*(k-1-\ell)}\big)*F_n^{*\ell}\big\|_{\phi_{\lambda'}}+\|F_n^{*(k-1)}-F^{*(k-1)}\|_{\phi_{\lambda'}}.
\end{eqnarray}
We start by considering the first summand in (\ref{eq proof convolution telesum last}). In view of (\ref{telescoping sum convolution}) it can be written as
\begin{eqnarray*}
    \lefteqn{\big\|\big((F_n+\varepsilon_nv_n)^{*(k-1-\ell)}-F_n^{*(k-1-\ell)}\big)*F_n^{*\ell}\big\|_{\phi_{\lambda'}}}\\
    & = & \big\|\big((F_n+\varepsilon_nv_n-F_n)*H_{k-1-\ell}(F_n+\varepsilon_nv_n,F_n)\big)*F_n^{*\ell}\big\|_{\phi_{\lambda'}}\\
    & = & \big\|\big(\varepsilon_nv_n*H_{k-1-\ell}(F_n+\varepsilon_nv_n,F_n)\big)*F_n^{*\ell}\big\|_{\phi_{\lambda'}}.
\end{eqnarray*}
Applying Lemma 2.3 of \cite{Pitts1994} with $f= \varepsilon_nv_n*H_{k-\ell-1}(F_n+\varepsilon_nv_n,F_n)$ and $H=F_n^{*\ell}$ we obtain
\begin{eqnarray}\label{eq proof convolution applying 2.3}
    \lefteqn{\big\|\big(\varepsilon_nv_n*H_{k-1-\ell}(F_n+\varepsilon_nv_n,F_n)\big)*F_n^{*\ell}\big\|_{\phi_{\lambda'}}}\nonumber\\
    & \le & 2^{\lambda'} \big\|\big(\varepsilon_nv_n*H_{k-\ell-1}(F_n+\varepsilon_nv_n,F_n)\big)\big\|_{\phi_{\lambda'}} \big(\|\eins_{[0,\infty)}\|F_n^{*\ell}\|_{\infty}-F_n^{*\ell}\|_{\phi_{\lambda'}}+\|F_n^{*\ell}\|_{\infty}\big) \nonumber\\
    & = & 2^{\lambda' }\big\|\big(\varepsilon_nv_n*H_{k-\ell-1}(F_n+\varepsilon_nv_n,F_n)\big)\|_{\phi_{\lambda'}}\big(\|\eins_{[0,\infty)}-F_n^{*\ell}\|_{\phi_{\lambda'}}+1\big) \nonumber \\
    & \le & 2^{\lambda'}\big\|\big(\varepsilon_nv_n*H_{k-\ell-1}(F_n+\varepsilon_nv_n,F_n)\big)\big\|_{\phi_{\lambda'}}\big\{(2^{\lambda'-1}\vee 1)\big(1+\ell^{\lambda' \vee 1} C_1 \big)+1\big\},
\end{eqnarray}
where we applied part (i) of Lemma \ref{lemma preceding qHD of compound}  to $\|\eins_{[0,\infty)}-F_n^{*\ell}\|_{\phi_{\lambda'}}$ to obtain the last inequality. Hence for the left-hand side of (\ref{eq proof convolution applying 2.3}) to go to zero as $n \rightarrow \infty$ it suffices to show that  $\|(\varepsilon_nv_n*H_{k-\ell-1}(F_n+\varepsilon_nv_n,F_n))\|_{\phi_{\lambda'}} \rightarrow 0$ as $n \rightarrow \infty$. The latter follows from
\begin{eqnarray}\label{eq proof convolution still not finsihed}
    \lefteqn{\big\|\big(\varepsilon_nv_n*H_{k-\ell-1}(F_n+\varepsilon_nv_n,F_n)\big)\big\|_{\phi_{\lambda'}}} \nonumber \\
    & \le & 2^{\lambda'} (k-\ell-1) \varepsilon_n \|v_n\|_{\phi_{\lambda'}}\big(1+2^{\lambda'}(2^{\lambda'-1} \vee 1)\big(2+((k-\ell-2))^{\lambda' \vee 1}C_2\big)\big),
\end{eqnarray}
where we applied part (ii) of Lemma \ref{lemma preceding qHD of compound} with $v=\varepsilon_nv_n$ to all summands in  $H_{k-\ell-1}(F_n+\varepsilon_nv_n,F_n)$. For every $k$ and $\ell\in\{0,\ldots,k-1\}$ this expression goes indeed to zero as $n \rightarrow \infty$, because, as mentioned before, $\|v_n\|_{\phi_{\lambda'}}$ is uniformly bounded in $n\in\N$, and we have $\varepsilon_n \rightarrow 0$. Next we consider the second summand in (\ref{eq proof convolution telesum last}). Applying (\ref{telescoping sum convolution}) to $F_n^{*(k-1)}$ and $F^{*(k-1)}$ and subsequently part (ii) of Lemma \ref{lemma preceding qHD of compound} to the summands in $H_{k-1}(F_n,F)$ we have
$$
    \|F_n^{*(k-1)}-F^{*(k-1)}\|_{\phi_{\lambda'}} \leq 2^{\lambda'}(k-1)\|F_n-F\|_{\phi_{\lambda'}} \big(1+2^{\lambda'}(2^{\lambda'-1} \vee 1)(2+((k-2))^{\lambda' \vee 1}C_2)\big).
$$
Clearly for every $k$ this term goes to zero 0 as $n \rightarrow \infty$, because $\|F_n-F\|_{\phi_{\lambda'}} \leq \|F_n-F\|_{\phi_{\lambda}} \rightarrow 0$ as $n \rightarrow \infty$ by assumption. This together with the fact that (\ref{eq proof convolution applying 2.3}) goes to zero 0 as $n \rightarrow \infty$ shows that (\ref{eq proof convolution telesum last}) goes to zero in  $\| \cdot \|_{\phi_{\lambda'}}$ as $n \rightarrow \infty$. Therefore, the second summand in (\ref{eq proof convoultion triangle three times}) goes to zero as $n \rightarrow \infty$.

It remains to consider the first term in (\ref{eq proof convoultion triangle three times}). We find
 \begin{eqnarray}\label{eq proof Pitts very last hopefully}
    \lefteqn{2^{\lambda'}\|\widetilde{v}-v\|_{\phi_{\lambda}}  \big(\|\eins_{[0,\infty)} - (F_n+\varepsilon_nv_n)^{*(k-1-\ell)}*F_n^{*\ell}\|_{\phi_{\lambda'}}+1\big)}\nonumber\\
    & \le &  2^{\lambda'}\|\widetilde{v}-v\|_{\phi_{\lambda'}}  \big(\|\eins_{[0,\infty)}- (F_n+\varepsilon_nv_n)^{*(k-1-\ell)}*F_n^{*\ell}\|_{\phi_{\lambda'}}+1\big)\nonumber\\
    & \le & 2^{\lambda'}\|\widetilde{v}-v\|_{\phi_{\lambda'}}  \big(\|\eins_{[0,\infty)}-F^{*(k-1)} + F^{*(k-1)} - (F_n+\varepsilon_nv_n)^{*(k-1-\ell)}*F_n^{*\ell}\|_{\phi_{\lambda'}}+1\big) \nonumber \\
    & \le & 2^{\lambda'}\|\widetilde{v}-v\|_{\phi_{\lambda'}}  \big(\|\eins_{[0,\infty)}-F^{*(k-1)}\|_{\phi_{\lambda'}} + \|F^{*(k-1)} - (F_n+\varepsilon_nv_n)^{*(k-1-\ell)}*F_n^{*\ell}\|_{\phi_{\lambda'}}+1\big) \nonumber \\
    & \le & 2^{\lambda'}\|\widetilde{v}-v\|_{\phi_{\lambda'}}  (2^{\lambda'-1} \vee 1)\Big(1+k^{\lambda \vee 1} \int |x|^{\lambda'}\,dF(x)\Big) \nonumber \\
    & &  +\,2^{\lambda'}\|\widetilde{v}-v\|_{\phi_{\lambda'}} \big(\big\|F^{*(k-1)} - (F_n+\varepsilon_nv_n)^{*(k-1-\ell)}*F_n^{*\ell}\big\|_{\phi_{\lambda'}}+1\big),
\end{eqnarray}
where for the last inequality we used Formula (2.4) of \cite{Pitts1994}. In the lines following (\ref{eq proof convolution telesum last}) we showed that
$\|F^{*(k-1)} - (F_n+\varepsilon_nv_n)^{*(k-1-\ell)}*F_n^{*\ell}\|_{\phi_{\lambda'}}$ goes to zero as $n \rightarrow \infty$ for every $k$ and $\ell\in\{0,\ldots,k-1\}$. Hence for every such $k$ and $\ell$,  it is uniformly bounded in $n\in\N$. Therefore we can make (\ref{eq proof Pitts very last hopefully}) arbitrarily small by making $\|\widetilde{v}-v\|_{\phi_{\lambda'}}$ small which, as mentioned above, is possible according to Lemma 4.2 of \cite{Pitts1994}. This finishes the proof.
\end{proof}


\subsection{Composition of Average Value at Risk functional and compound distribution functional}\label{Composition functional}

Here we consider the composition of the Average Value at Risk functional ${\cal R}_\alpha$ defined in (\ref{def avarf}) and the compound distribution functional ${\cal C}_p$ defined in (\ref{def codifu}). As a consequence of Propositions \ref{proposition AVaR} and \ref{QHD of CF} we obtain the following Corollary \ref{QHD of Composition}. Note that, for any $\lambda>1$, Lemma 2.2 in \cite{Pitts1994} yields ${\cal C}_p(\mathbf{F}_{\phi_{\lambda}}) \subseteq \mathbf{F}_1$ so that the composition ${\cal R}_\alpha\circ{\cal C}_{p}$ is well defined on $\F_{\phi_\lambda}$.

\begin{corollary}\label{QHD of Composition}
Assume that $\sum_{k=1}^{\infty} p_k\,k^{(1+\lambda)\vee 2} < \infty$. Let $\lambda>1$, $F \in \mathbf{F}_{\phi_{\lambda}}$, and assume that ${\cal C}_p(F)$ takes the value $1-\alpha$ only once. Then the map $T_{\alpha,p}:={\cal R}_\alpha\circ{\cal C}_{p}: \mathbf{F}_{\phi_{\lambda}} (\subseteq \mathbf{D})  \rightarrow \R$ is uniformly quasi-Hadamard differentiable at $F$ tangentially to $\bD_{\phi_{\lambda}}\langle\bD_{\phi_{\lambda}}\rangle$, and the uniform quasi-Hadamard derivative $\dot T_{\alpha,p;F}:\bD_{\phi_\lambda}\rightarrow\R$ is given by $\dot T_{\alpha,p;F}=\dot{\cal R}_{\alpha;{\cal C}_p(F)} \circ \dot{{\cal C}}_{p;F}$, i.e.
$$
    \dot T_{\alpha,p;F}(v)=\int g_\alpha'({\cal C}_p(F)(x))(v*H_{p,F})(x)\,dx\qquad\mbox{for all }v\in\bD_{\phi_\lambda}
$$
with $g_\alpha'$ and $v*H_{p,F}$ as in Proposition \ref{proposition AVaR} and \ref{QHD of CF}, respectively.
\end{corollary}

\begin{proof}
We intend to apply Lemma \ref{lemma chain rule} to $H={\cal C}_p:\F_{\phi_{\lambda}}\rightarrow \F_{1}$ and $\widetilde H={\cal R}_\alpha: \F_{1} \rightarrow\R$. To verify that the assumptions of the lemma are fulfilled, we first recall from the comment directly before Corollary \ref{QHD of Composition} that ${\cal C}_p(\mathbf{F}_{\phi_{\lambda}}) \subseteq \mathbf{F}_1$. It remains to show that the assumptions (a)--(c) of Lemma \ref{lemma chain rule} are fulfilled. According to Proposition \ref{QHD of CF} we have that for every $\lambda' \in (1,\lambda)$ the functional ${\cal C}_p$ is uniformly quasi-Hadamard differentiable at $F$ tangentially to $\bD_{\phi_{\lambda}}\langle\bD_{\phi_{\lambda}}\rangle$ with trace $\bD_{\phi_{\lambda'}}$, which is the first part of assumption (b). The second part of assumption (b) means $\dot{\cal C}_{p,F}(\bD_{\phi_{\lambda}}) \subseteq \bD_{\phi_{\lambda'}}$ and follows from
\begin{eqnarray*}
    \|\dot{\cal C}_{p;F}(v)\|_{\phi_{\lambda'}}
    & = & \Big\| v*\sum_{k=1}^{\infty} p_k\,k F^{*(k-1)}\Big\|_{\phi_{\lambda'}} \\
    & \le & 2^{\lambda'} \|v\|_{\phi_{\lambda'}} \sum_{k=1}^{\infty} p_k\,k\Big(1+(2^{\lambda'-1} \vee 1)\Big(1+k^{\lambda' \vee 1}\int |x|^{\lambda'}\,dF(x)\Big)\Big)
\end{eqnarray*}
(for which we applied Lemma 2.3 and Inequality (2.4) in \cite{Pitts1994}), the convergence of the latter series (which holds by assumption), and  $\|v\|_{\phi_{\lambda'}}\leq \|v\|_{\phi_{\lambda}}<\infty$. Further, it follows from Proposition \ref{proposition AVaR} that the map ${\cal R}_{\alpha}$ is uniformly quasi-Hadamard differentiable tangentially to $\bD_{\phi_{\lambda'}}\langle\bD_{\phi_{\lambda'}}\rangle$ at every distribution function of $\F_{\phi_{\lambda'}}$ that takes the value $1 - \alpha$ only once. This is assumption (c) of Lemma \ref{lemma chain rule}.

It remains to show that also assumption (a) of Lemma \ref{lemma chain rule} holds true. In the present setting assumption (a) means that for every sequence $(F_n) \subseteq\F_{\phi_\lambda}$ with $\|F_n-F\|_{\phi_{\lambda}} \rightarrow 0$ we have ${\cal C}_p(F_n)\to{\cal C}_p(F)$ pointwise. We will show that we even have $\|{\cal C}_p(F_n)-{\cal C}_p(F)\|_{\phi_{\lambda'}} \rightarrow 0$. So let $(F_n)\subseteq \F_{\phi_\lambda}$. Then
\begin{eqnarray*}
    \|{\cal C}_p(F_n)-{\cal C}_p(F)\|_{\phi_{\lambda'}}
    & = & \Big\|\sum_{k=1}^{\infty} p_k(F_n^{*k}-F^{*k})\Big\|_{\phi_{\lambda'}}\\
    & = & \Big\| (F_n-F)* \sum_{k=1}^{\infty} p_k H_k(F_n,F)\Big\|_{\phi_{\lambda'}}\\
    & \le & 2^{\lambda'}\|F_n-F\|_{\phi_{\lambda'}} \sum_{k=1}^{\infty} p_k\, k \big(1+2^{\lambda'}(2^{\lambda'-1} \vee 1)\big(2+(k-1)^{\lambda' \vee 1}C_2\big)\big),
\end{eqnarray*}
where we used (\ref{telescoping sum convolution}) for the second ``$=$'' and applied part (ii) of Lemma \ref{lemma preceding qHD of compound} to the summands of $H_k$ to obtain the latter inequality. Since the series converges, we obtain $\|{\cal C}_p(F_n)-{\cal C}_p(F)\|_{\phi_{\lambda'}} \rightarrow 0$ when assuming $\|F_n-F\|_{\phi_{\lambda}} \rightarrow 0$.
\end{proof}

As an immediate consequence of Theorem \ref{modified delta method for the bootstrap - II}, Examples \ref{iid bootstrap} and \ref{block bootstrap}, and Corollary \ref{QHD of Composition} we obtain the following corollary.

\begin{corollary}
Let $F$, $\widehat F_n$, $\widehat F_n^*$, $\widehat C_n$, and $B_F$ be as in Example \ref{iid bootstrap} (S1.\ or S2.) or as in Example \ref{block bootstrap} respectively, and assume that the assumptions discussed in Example \ref{iid bootstrap} or in Example \ref{block bootstrap} respectively are fulfilled for some weight function $\phi$ with $\int 1/\phi(x)\,dx<\infty$ (in particular $F\in\F_1$). Then
$$
    \sqrt{n}\big(T_{\alpha,p}(\widehat F_n)-T_{\alpha,p}(F)\big)\,\leadsto\,\dot T_{\alpha,p;F}(B_F)\qquad\mbox{in $(\R,{\cal B}(\R))$}
$$
and
$$
    \sqrt{n}\big(T_{\alpha,p}(\widehat F_n^*(\omega,\cdot))-T_{\alpha,p}(\widehat C_n(\omega))\big)\,\leadsto\,\dot T_{\alpha,p;F}(B_F)\qquad\mbox{in $(\R,{\cal B}(\R))$},\qquad\mbox{$\pr$-a.e.\ $\omega$}.
$$
\end{corollary}


\appendix

\section{Convergence in distribution$^\circ$}\label{Appendix Weak Convergence}

Let $(\bE,d)$ be a metric space and ${\cal B}^\circ$ be the $\sigma$-algebra on $\bE$ generated by the open balls $B_r(x):=\{y\in\bE:d(x,y)<r\}$, $x\in\bE$, $r>0$. We will refer to ${\cal B}^\circ$ as {\em open-ball $\sigma$-algebra}. If $(\bE,d)$ is separable, then ${\cal B}^\circ$ coincides with the Borel $\sigma$-algebra ${\cal B}$. If $(\bE,d)$ is not separable, then ${\cal B}^\circ$ might be strictly smaller than ${\cal B}$ and thus a continuous real-valued function on $\bE$ is not necessarily $({\cal B}^\circ,{\cal B}(\R))$-measurable. Let $C_{\rm b}^\circ$ be the set of all bounded, continuous and $({\cal B}^\circ,{\cal B}(\R))$-measurable real-valued functions on $\bE$, and ${\cal M}_1^\circ$ be the set of all probability measures on $(\bE,{\cal B}^\circ)$.

Let $X_n$ be an $(\bE,{\cal B}^\circ)$-valued random variable on some probability space $(\Omega_n,{\cal F}_n,\pr_n)$ for every $n\in\N_0$. Then the sequence $(X_n)=(X_n)_{n\in\N}$ is said to {\em converge in distribution$^\circ$} to $X_0$ if
$$
    \int f\,d\pr\circ X_n^{-1}\,\longrightarrow\,\int f\,d\pr_0\circ X_0^{-1}\qquad\mbox{for all }f\in C_{\rm b}^\circ.
$$
In this case, we write $X_n\leadsto^\circ X_0$. This is the same as saying that the sequence $(\pr_n\circ X_n^{-1})$ converges to $\pr_0\circ X_0^{-1}$ in the weak$^\circ$ topology on ${\cal M}_1^\circ$; for details see the Appendix A of \cite{BeutnerZaehle2016}. It is worth mentioning that two probability measures $\mu,\nu\in{\cal M}_1^\circ$ coincide if $\mu[\bE_0]=\nu[\bE_0]=1$ for some separable $\bE_0\in{\cal B}^\circ$ and $\int f\,d\mu=\int f\,d\nu$ for all uniformly continuous $f\in C_{\rm b}^\circ$; see, for instance, \cite[Theorem 6.2]{Billingsley1999}.

In the Appendices A--C in \cite{BeutnerZaehle2016} several properties of  convergence in distribution$^\circ$ (and weak$^\circ$ convergence) have been discussed. The following two subsections complement this discussion.


\subsection{Slutsky-type results for the open-ball $\sigma$-algebra}\label{Slutzky type results}

For a sequence $(X_n)$ of $(\bE,{\cal B}^\circ)$-valued random variables that are all defined on the same probability space $(\Omega,{\cal F},\pr)$, the sequence $(X_n)$ is said to {\em converge in probability$^\circ$} to $X_0$ if the mappings $\omega\mapsto d(X_n(\omega),X_0(\omega))$, $n\in\N$, are $({\cal F},{\cal B}(\R_+))$-measurable and satisfy
\begin{equation}\label{def conv in prob - eq}
    \lim_{n\to\infty}\pr[d(X_n,X_0)\ge\varepsilon]=0\quad\mbox{ for all }\varepsilon>0.
\end{equation}
In this case, we write $X_n\rightarrow^{{\sf p},\circ} X_0$. The superscript $^\circ$ points to the fact that measurability of the mapping $\omega\mapsto d(X_n(\omega),X_0(\omega))$ is a requirement of the definition (and not automatically valid). Note however that in the specific situation where $X_0\equiv x_0$ for some $x_0\in\bE$, measurability of the mapping $\omega\mapsto d(X_n(\omega),X_0(\omega))$ does hold; cf.\ Lemma B.3 in \cite{BeutnerZaehle2016}. Also note that the measurability always hold when $(\bE,d)$ is separable; in this case we also write $\rightarrow^{{\sf p}}$ instead of $\rightarrow^{{\sf p},\circ}$.

\begin{theorem}\label{Slutzky}
Let $(X_n)$ and $(Y_n)$ be two sequences of $(\bE,{\cal B}^\circ)$-valued random variables on a common probability space $(\Omega,{\cal F},\pr)$, and assume that the mapping $\omega\mapsto d(X_n(\omega),Y_n(\omega))$ is $({\cal F},{\cal B}(\R_+))$-measurable for every $n\in\N$. Let $X_0$ be an $(\bE,{\cal B}^\circ)$-valued random variable on some probability space $(\Omega_0,{\cal F}_0,\pr_0)$ with $\pr_0[X_0\in\bE_0]=1$ for some separable $\bE_0\in{\cal B}^\circ$. Then $X_n\leadsto^\circ X_{0}$ and $d(X_n,Y_n)\rightarrow^{{\sf p}} 0$ together imply $Y_n\leadsto^\circ X_{0}$.
\end{theorem}

\begin{proof}
In view of $X_n\leadsto^\circ X$, we obtain for every fixed $f\in{\rm BL}_1^\circ$
\begin{eqnarray*}
    \lefteqn{\limsup_{n\to\infty}\Big|\int f\,d\pr_{Y_n}-\int f\,d\pr_{X_0}\Big|}\\
    & \le & \limsup_{n\to\infty}\Big|\int f\,d\pr_{Y_n}-\int f\,d\pr_{X_n}\Big|+\limsup_{n\to\infty}\Big|\int f\,d\pr_{X_n}-\int f\,d\pr_{X_0}\Big|\\
    & \le & \limsup_{n\to\infty}\int |f(Y_n)-f(X_n)|\,d\pr.
\end{eqnarray*}
Since $f$ lies in ${\rm BL}_1^\circ$ and we assumed $d(X_n,Y_n)\rightarrow^{{\sf p}} 0$, we also have
\begin{eqnarray*}
    \limsup_{n\to\infty}\int |f(Y_n)-f(X_n)|\,d\pr
    & \le & \limsup_{n\to\infty}\Big|\int |f(Y_n)-f(X_n)|\eins_{\{d(X_n,Y_n)\ge\varepsilon\}}\,d\pr\,+\,2\varepsilon\\
    & \le & 2\limsup_{n\to\infty}\pr[d(X_n,Y_n)\ge\varepsilon]\,+\,2\varepsilon
\end{eqnarray*}
for every $\varepsilon>0$. Thus $\limsup_{n\to\infty}\int |f(Y_n)-f(X_n)|\,d\pr=0$ which together with the Portmanteau theorem (in the form of \cite[Theorem A.4]{BeutnerZaehle2016}) implies the claim.
\end{proof}

Set $\overline{\bE}:=\bE\times\bE$ and let $\overline{{\cal B}}^\circ$ be the $\sigma$-algebra on $\overline{\bE}$ generated by the open balls w.r.t.\ the metric
$$
    \overline{d}((x_1,x_2),(y_1,y_2)):=\max\{d(x_1,y_1);d(x_2,y_2)\}.
$$
Recall that $\overline{{\cal B}}^\circ\subseteq{\cal B}^\circ\otimes{\cal B}^\circ$, where the inclusion may be strict.

\begin{corollary}\label{Slutzky corollary 1}
Let $(X_n)$ and $(Y_n)$ be two sequences of $(\bE,{\cal B}^\circ)$-valued random variables on a common probability space $(\Omega,{\cal F},\pr)$. Let $X_0$ be an $(\bE,{\cal B}^\circ)$-valued random variable on some probability space $(\Omega_0,{\cal F}_0,\pr_0)$ with $\pr_0[X_0\in\bE_0]=1$ for some separable $\bE_0\in{\cal B}^\circ$. Let $y_0\in\bE_0$. Let $(\widetilde\bE,\widetilde{d})$ be a metric space equipped with the corresponding open-ball $\sigma$-algebra $\widetilde{\cal B}^\circ$. Then $X_n\leadsto^\circ X_0$ and $Y_n\rightarrow^{{\sf p},\circ} y_0$ together imply
\begin{itemize}
    \item[(i)] $(X_n,Y_n)\leadsto^\circ (X_0,y_0)$.
    \item[(ii)]$h(X_n,Y_n)\leadsto^\circ h(X_0,y_0)$ for every continuous and $(\overline{\cal B}^\circ,\widetilde{\cal B}^\circ)$-measurable $h:\overline\bE\rightarrow\widetilde\bE$.
\end{itemize}
\end{corollary}

\begin{proof}
Assertion (ii) is an immediate consequence of assertion (i) and the Continuous Mapping theorem in the form of \cite[Theorem 6.4]{Billingsley1999}; take into account that  $(X_0,y_0)$ takes values only in $\overline\bE_0:=\bE_0\times\bE_0$ and that $\bE_0\times\bE_0$ is separable w.r.t.\ $\overline{d}$. Thus it suffices to show assertion (i). First note that we have
\begin{equation}\label{Slutzky corollary 1 - PROOF - 10}
    (X_n,y_0)\,\leadsto^\circ\,(X_0,y_0).
\end{equation}
Indeed, for every $f\in\overline C_{\rm b}^\circ$ (with $\overline C_{\rm b}^\circ$ the set of all bounded, continuous and $(\overline{\cal B}^\circ,{\cal B}(\R))$-measurable real-valued functions on $\overline\bE$) we have $\lim_{n\to\infty}\int f(X_n,y_0)\,d\pr=\int f(X_0,y_0)\,d\pr_0$ by the assumption $X_n\leadsto^\circ X_0$ and the fact that the mapping $x\mapsto f(x,y_0)$ lies in $C_{\rm b}^\circ$ (the latter was shown in the proof of Theorem 3.1 in \cite{BeutnerZaehle2016}).

Second, the distance $\overline{d}((X_n,Y_n),(X_n,y_0))=d(Y_n,y_0)$ is $({\cal F},{\cal B}(\R_+))$-measurable for every $n\in\N$, because $Y_n$ is $({\cal F},{\cal B}^\circ)$-measurable and $x\mapsto d(x,y_0)$ is $({\cal B}^\circ,{\cal B}(\R))$-measurable (due to Lemma B.3 in \cite{BeutnerZaehle2016}). Along with $Y_n\rightarrow^{{\sf p},\circ} y_0$ we obtain in particular that $\overline{d}((X_n,Y_n),(X_n,y_0))\rightarrow^{{\sf p}}0$. Together with (\ref{Slutzky corollary 1 - PROOF - 10}) and Theorem \ref{Slutzky} (applied to $X_n':=(X_n,y_0)$, $X_0':=(X_0,y_0)$, $Y_n':=(X_n,Y_n)$) this implies $(X_n,Y_n)\leadsto^\circ(X_0,y_0)$; take into account again that $(X_0,y_0)$ takes values only in $\overline\bE_0:=\bE_0\times\bE_0$ and that $\bE_0\times\bE_0$ is separable w.r.t.\ $\overline{d}$.
\end{proof}

\begin{corollary}\label{Slutzky corollary 2}
Let $(\bE,\|\cdot\|_\bE)$ be a normed vector space and $d$ be the induced metric defined by $d(x_1,x_2):=\|x_1-x_2\|_\bE$. Let $(X_n)$ and $(Y_n)$ be two sequences of $(\bE,{\cal B}^\circ)$-valued random variables on a common probability space $(\Omega,{\cal F},\pr)$. Let $X_0$ be an $(\bE,{\cal B}^\circ)$-valued random variable on some probability space $(\Omega_0,{\cal F}_0,\pr_0)$ with $\pr_0[X_0\in\bE_0]=1$ for some separable $\bE_0\in{\cal B}^\circ$. Let $y_0\in\bE_0$. Assume that the map $h:\overline{\bE}\rightarrow\bE$ defined by $h(x_1,x_2):=x_1+x_2$ is $(\overline{{\cal B}}^\circ,{\cal B}^\circ)$-measurable. Then $X_n\leadsto^\circ X_0$ and $Y_n\rightarrow^{{\sf p},\circ} y_0$ together imply $X_n+Y_n\leadsto^\circ X_0+y_0$.
\end{corollary}

\begin{proof}
The assertion is an immediate consequence of Corollary \ref{Slutzky corollary 1} and the fact that $h$ is clearly continuous (w.r.t.\ $\overline{d}$ and the Euclidean distance $|\cdot|$).
\end{proof}


\subsection{Delta-method and chain rule for uniformly quasi-Hadamard differentiable maps}\label{Delta-method for uniformly QHD maps}

Now assume that $\bE$ is a subspace of a vector space $\V$. Let $\|\cdot\|_\bE$ be a norm on $\bE$ and assume that the metric $d$ is induced by $\|\cdot\|_\bE$. Let $\widetilde\V$ be another vector space and $\widetilde\bE\subseteq\widetilde\V$ be any subspace. Let $\|\cdot\|_{\widetilde{\bE}}$ be a norm on $\widetilde{\bE}$ and $\widetilde{\cal B}^\circ$ be the corresponding open-ball $\sigma$-algebra on $\widetilde\bE$. Let $0_{\widetilde\bE}$ denote the null in $\widetilde\bE$. Moreover, let  $\overline{\widetilde\bE}:=\widetilde\bE\times\widetilde\bE$ and $\overline{\widetilde{\cal B}^\circ}$ be the $\sigma$-algebra on $\overline{\widetilde\bE}$ generated by the open balls w.r.t.\ the metric $\overline{\widetilde d}((\widetilde x_1,\widetilde x_2),(\widetilde y_1,\widetilde y_2)):=\max\{\|\widetilde x_1-\widetilde y_1\|_{\widetilde\bE};\|\widetilde x_2-\widetilde y_2\|_{\widetilde\bE}\}$.

Let $(\Omega_n,{\cal F}_n,\pr_n)$ be a probability space and $\widehat T_n:\Omega_n\rightarrow \V$ be any map for every $n\in\N$. Recall that $\leadsto^\circ$ and $\rightarrow^{{\sf p},\circ}$ refer to convergence in distribution$^\circ$ and convergence in probability$^\circ$, respectively. Moreover recall Definition \ref{definition quasi hadamard} of quasi-Hadamard differentiability.

\begin{theorem}\label{modified delta method}
Let $H:\V_H\to\widetilde\bE$ be a map defined on some $\V_H \subseteq\V$. Let $\bE_0\in{\cal B}^\circ$ be some $\|\cdot\|_{\bE}$-separable subset of $\bE$. Let $(\theta_n)\subseteq\V_H$ and define the singleton set ${\cal S}:=\{(\theta_n)\}$. Let $(a_n)$ be a sequence of positive real numbers tending to $\infty$, and consider the following conditions:
\begin{itemize}
    \item[(a)] $\widehat T_n$ takes values only in $\V_H$.
    \item[(b)] $a_n(\widehat T_n-\theta_n)$ takes values only in $\bE$, is $({\cal F}_n,{\cal B}^\circ)$-measurable and satisfies
    \begin{eqnarray}\label{modified delta method - assumption}
         a_n(\widehat T_n-\theta_n)\,\leadsto^\circ\,\xi\qquad\mbox{in $(\bE,{\cal B}^\circ,\|\cdot\|_{\bE})$}
     \end{eqnarray}
    for some $(\bE,{\cal B}^\circ)$-valued random variable $\xi$ on some probability space $(\Omega_0,{\cal F}_0,\pr_0)$ with $\xi(\Omega_0)\subseteq\bE_0$.
    \item[(c)] $a_n(H(\widehat T_n)-H(\theta_n))$ takes values only in $\widetilde\bE$ and is $({\cal F}_n,\widetilde{\cal B}^\circ)$-measurable.
    \item[(d)] The map $H$ is uniformly quasi-Hadamard differentiable w.r.t.\ ${\cal S}$ tangentially to $\bE_0\langle\bE\rangle$ with trace $\widetilde{\bE}$ and uniform quasi-Hadamard derivative $\dot H_{\cal S}:\bE_0\rightarrow\widetilde\bE$. 
    \item[(e)] $(\Omega_n,{\cal F}_n,\pr_n)=(\Omega,{\cal F},\pr)$ for all $n\in\N$.
    \item[(f)] The uniform quasi-Hadamard derivative $\dot H_{\cal S}$ can be extended to $\bE$ such that the extension $\dot H_{\cal S}:\bE\rightarrow\widetilde\bE$ is continuous at every point of $\bE_0$ and $({\cal B}^\circ,\widetilde{\cal B}^\circ)$-measurable.
    \item[(g)] The map $h:\overline{\widetilde\bE}\rightarrow\widetilde\bE$ defined by $h(\widetilde x_1,\widetilde x_2):=\widetilde x_1-\widetilde x_2$ is $(\overline{\widetilde{\cal B}^\circ},\widetilde{\cal B}^\circ)$-measurable.
\end{itemize}
Then the following two assertions hold:
\begin{itemize}
    \item[(i)] If conditions (a)--(d) hold true, then $\dot H_{\cal S}(\xi)$ is $({\cal F}_0,\widetilde{\cal B}^\circ)$-measurable and
    $$
        a_n\big(H(\widehat T_n)-H(\theta_n)\big)\,\leadsto^\circ\,\dot H_{\cal S}(\xi)\qquad\mbox{in $(\widetilde\bE,\widetilde{\cal B}^\circ,\|\cdot\|_{\widetilde\bE})$}.
    $$
    \item[(ii)] If conditions (a)--(g) hold true, then
    \begin{equation}\label{modified fct delta eq - ii}
        a_n\big(H(\widehat T_n)-H(\theta_n)\big)-\dot H_{\cal S}\big(a_n(\widehat T_n-\theta_n)\big)\,\rightarrow^{{\sf p},\circ}\,0_{\widetilde\bE}\qquad\mbox{in $(\widetilde\bE,\|\cdot\|_{\widetilde\bE})$}.
    \end{equation}
\end{itemize}
\end{theorem}

\begin{proof}
The proof is very similar to the proof of Theorem C.4 in \cite{BeutnerZaehle2016}.

(i): For every $n\in\N$, let $\bE_{n}:=\{x_n\in\bE:\theta_n+a_n^{-1}x_n \in \V_H\}$ and define the map $h_n:\bE_{n} \rightarrow \widetilde\bE$ by
$$
    h_n(x_n)\,:=\,\frac{H(\theta_n+a_n^{-1}x_n)-H(\theta_n)}{a_n^{-1}}\,.
$$
Moreover, define the map $h_0: \bE_0 \rightarrow \widetilde\bE$ by
$$
    h_0(x)\,:=\,\dot H_{\cal S}(x).
$$
Now, the claim would follow by the extended Continuous Mapping theorem in the form of Theorem C.1 in \cite{BeutnerZaehle2016} applied to the functions $h_n$, $n\in\N_0$, and the random variables $\xi_n:=a_n(\widehat T_n-\theta_n)$, $n\in\N$, and $\xi_0:=\xi$ if we can show that the assumptions of Theorem C.1 in \cite{BeutnerZaehle2016} are satisfied. First, by assumption (a) and the last part of assumption (b) we have $\xi_n(\Omega_n)\subseteq\bE_n$ and $\xi_0(\Omega_0)\subseteq\bE_0$. Second, by assumption (c) we have that $h_n(\xi_n)=a_n(H(\widehat T_n)-H(\theta_n))$ is $({\cal F}_n,\widetilde{\cal B}^\circ)$-measurable. Third, the map $h_0$ is continuous by the definition of the quasi-Hadamard derivative. Thus $h_0$ is $({\cal B}_0^\circ,\widetilde{\cal B}^\circ)$-measurable, because the trace $\sigma$-algebra ${\cal B}_0^\circ:={\cal B}^\circ\cap\bE_0$ coincides with the Borel $\sigma$-algebra on $\bE_0$ (recall that $\bE_0$ is separable). In particular, $\dot H_{\cal S}(\xi)$ is $({\cal F}_0,\widetilde{\cal B}^\circ)$-measurable. Fourth, condition (a) of Theorem C.1 in \cite{BeutnerZaehle2016} holds by assumption (b). Fifth, condition (b) of Theorem C.1 in \cite{BeutnerZaehle2016} is ensured by assumption (d).

(ii): For every $n\in\N$, let $\bE_n$ and $h_n$ be as above and define the map $\overline h_n: \bE_{n} \rightarrow \overline{\widetilde\bE}$ by
$$
    \overline h_n(x_n)\,:=\,(h_n(x_n),\dot H_{\cal S}(x_n)).
$$
Moreover, define the map $\overline h_0: \bE_0 \rightarrow\overline{\widetilde\bE}$ by
$$
    \overline h_0(x)\,:=\,(h_0(x),\dot H_{\cal S}(x))\,=\,(\dot H_{\cal S}(x),\dot H_{\cal S}(x)).
$$
We will first show that
\begin{equation}\label{modified delta method PROOF 10}
    \overline h_n(a_n(X_n-x))\,\leadsto^\circ\, \overline h_0(X_0)\qquad\mbox{in $(\overline{\widetilde{\bE}},\overline{\widetilde{\cal B}^\circ},\overline{\widetilde d})$}.
\end{equation}
For (\ref{modified delta method PROOF 10}) it suffices to show that the assumption of the extended Continuous Mapping theorem in the form of Theorem C.1 in \cite{BeutnerZaehle2016} applied to the functions $\overline h_n$ and $\xi_n$ (as defined above) are satisfied. The claim then follows by Theorem C.1 in \cite{BeutnerZaehle2016}. First, we have already observed that $\xi_n(\Omega_n)\subseteq\bE_n$ and $\xi_0(\Omega_0)\subseteq\bE_0$. Second, we have seen in the proof of part (i) that $h_n(\xi_n)$ is $({\cal F}_{n},\widetilde{\cal B}^\circ)$-measurable, $n \in \N$. By assumption (f) the extended map $\dot H_{\cal S}:\bE\rightarrow\widetilde\bE$ is $({\cal B}^\circ,\widetilde{\cal B}^\circ)$-measurable, which implies that $\dot H_{\cal S}(\xi_n)$ is $({\cal F}_{n},\widetilde{\cal B}^\circ)$-measurable. Thus, $\overline h_n(\xi_n)=(h_n(\xi_n),\dot H_{\cal S}(\xi_n))$ is $({\cal F}_{n},\widetilde{\cal B}^\circ\otimes\widetilde{\cal B}^\circ)$-measurable (to see this note that, in view of $\widetilde{\cal B}^\circ\otimes\widetilde{\cal B}^\circ=\sigma(\pi_1,\pi_2)$ for the coordinate projections $\pi_1,\pi_2$ on $\overline{\widetilde E}={\widetilde E}\times{\widetilde E}$, Theorem 7.4 of \cite{Bauer2001} shows that the map $(h_n(\xi_n),\dot H_{\cal S}(\xi_n))$ is $({\cal F}_{n},\widetilde{\cal B}^\circ\otimes\widetilde{\cal B}^\circ)$-measurable if and only if the maps $h_n(\xi_n)=\pi_1\circ(h_n(\xi_n),\dot H_{\cal S}(\xi_n))$ and $\dot H_{\cal S}(\xi_n)=\pi_2\circ(h_n(\xi_n),\dot H_{\cal S}(\xi_n))$ are $({\cal F}_n,{\widetilde{\cal B}}^\circ)$-measurable). In particular, the map $\overline h_n(\xi_n)=(h_n(\xi_n),\dot H_{\cal S}(\xi_n))$ is $({\cal F}_{n},\overline{\widetilde{\cal B}^\circ})$-measurable, $n\in\N$. Third, we have seen in the proof of part (i) that the map $h_0=\dot H_{\cal S}$ is $({\cal B}_0^\circ,\widetilde{\cal B}^\circ)$-measurable. Thus the map $\overline h_0$ is $({\cal B}_0^\circ,\widetilde{\cal B}^\circ\otimes\widetilde{\cal B}^\circ)$-measurable (one can argue as above) and in particular $({\cal B}_0^\circ,\overline{\widetilde{\cal B}^\circ})$-measurable. Fourth, condition (a) of Theorem C.1 in \cite{BeutnerZaehle2016} holds by assumption (b). Fifth, condition (b) of Theorem C.1 in \cite{BeutnerZaehle2016} is ensured by assumption (d) and the continuity of the extended map $\dot H_{\cal S}$ at every point of $\bE_0$ (recall assumption (f)). Hence, (\ref{modified delta method PROOF 10}) holds.

By assumption (g) and the ordinary Continuous Mapping theorem (cf.\ \cite[Theorem 6.4]{Billingsley1999}) applied to (\ref{modified delta method PROOF 10}) and the map $h:\overline{\widetilde\bE}\rightarrow\widetilde\bE$, $(\widetilde x_1,\widetilde x_2)\mapsto\widetilde x_1-\widetilde x_2$, we now have
$$
    h_n(a_n(\widehat T_n-\theta_n))-\dot H_{\cal S}(a_n(\widehat T_n-\theta_n))\,\leadsto^\circ\,\dot H_{\cal S}(\xi)-\dot H_{\cal S}(\xi),
$$
i.e.
$$
    a_n\big(H(\widehat T_n)-H(\theta_n)\big)-\dot H_{\cal S}\big(a_n(\widehat T_n-\theta_n)\big)\,\leadsto^\circ\,0_{\widetilde\bE}.
$$
By Proposition B.4 in \cite{BeutnerZaehle2016} we can conclude (\ref{modified fct delta eq - ii}).
\end{proof}

The following lemma provides a chain rule for uniformly quasi-Hadamard differentiable maps (a similar chain rule with different ${\cal S}$ was found in \cite{Varron2015}). To formulate the chain rule let $\widetilde{\widetilde{\V}}$ be a further vector space and $\widetilde{\widetilde{\bE}}\subseteq\widetilde{\widetilde{\V}}$ be a subspace equipped with a norm $\|\cdot\|_{\widetilde{\widetilde{\bE}}}$.

\begin{lemma}\label{lemma chain rule}
Let $H: \V_H\rightarrow \widetilde\V_{\widetilde{H}}$ and $\widetilde{H}: \widetilde\V_{\widetilde{H}}\rightarrow \widetilde{\widetilde{\V}}$ be maps defined on subsets $\V_H\subseteq\V$ and $\widetilde\V_{\widetilde H}\subseteq\widetilde\V$ such that $H(\V_H) \subseteq \widetilde{\V}_{\widetilde{H}}$. Let $\bE_0$ and $\widetilde\bE_0$ be subsets of $\bE$ and $\widetilde\bE$ respectively. Let $\mathcal{S}$ and $\widetilde{\mathcal{S}}$ be sets of sequences in $\V_{H}$ and $\widetilde{\V}_{\widetilde{H}}$ respectively, and assume that the following three assertions hold.
\begin{itemize}
    \item[(a)] For every $(\theta_n) \in \mathcal{S}$ we have $(H(\theta_n)) \in \widetilde{\mathcal{S}}$.
    \item[(b)] $H$ is uniformly quasi-Hadamard differentiable w.r.t.\ $\mathcal{S}$ tangentially to $\bE_0\langle\bE\rangle$ with trace $\widetilde\bE$ and uniform quasi-Hadamard derivative $\dot{H}_{\mathcal{S}}:\bE_0\rightarrow\widetilde\bE$, and we have $\dot{H}_{\mathcal{S}}(\bE_0) \subseteq \widetilde{\bE}_0$.
    \item[(c)] $\widetilde{H}$ is uniformly quasi-Hadamard differentiable w.r.t.\ $\widetilde{\mathcal{S}}$ tangentially to $ \widetilde{\bE}_0\langle\widetilde{\bE}\rangle$ with trace $\widetilde{\widetilde{\bE}}$ and uniform quasi-Hadamard derivative $\dot{\widetilde{H}}_{\widetilde{\mathcal{S}}}:\widetilde\bE_0\rightarrow\widetilde{\widetilde{\bE}}$.
\end{itemize}
Then the map $T:=\widetilde{H} \circ H : \V_H \rightarrow \widetilde{\widetilde{\V}}$ is uniformly quasi-Hadamard differentiable w.r.t.\ $\mathcal{S}$ tangentially to $\bE_0\langle\bE\rangle$ with trace $\widetilde{\widetilde{\bE}}$, and the uniform quasi-Hadamard derivative $\dot T_\mathcal{S}$ is given by $\dot T_\mathcal{S}=\dot{\widetilde{H}}_{\widetilde{\mathcal{S}}}\circ \dot{H}_{\mathcal{S}}$.
\end{lemma}

\begin{proof}
Obviously, since $H(\V_H)\subseteq \widetilde{\V}_{\widetilde{H}}$ and $\widetilde{H}$ is associated with trace $\widetilde{\widetilde{\bE}}$, the map $\widetilde{H} \circ H$ can also be associated with trace $\widetilde{\widetilde{\bE}}$.

Now let $((\theta_n),x,(x_{n}),(\varepsilon_n))$ be a quadruple with $(\theta_{n})\in{\cal S}$, $x\in\bE_0$, $(x_{n})\subseteq\bE$ satisfying $\|x_{n}-x\|_{\bE}\to 0$ as well as $(\theta_n+\varepsilon_nx_{n})\subseteq\V_H$, and $(\varepsilon_n)\subseteq(0,\infty)$ satisfying $\varepsilon_n\to 0$. Then
\begin{eqnarray*}\label{eq cahin rule 1}
    \lefteqn{\Big\|\dot{\widetilde{H}}_{\widetilde{\mathcal{S}}}(\dot{H}_{\mathcal{S}}(x))- \frac{\widetilde{H}(H(\theta_{n}+\varepsilon_nx_{n}))-\widetilde{H}(H(\theta_n))}{\varepsilon_n}\Big\|_{\widetilde{\widetilde{\bE}}}} \nonumber \\
    & = & \Big\|\dot{\widetilde{H}}_{\widetilde{\mathcal{S}}}(\dot{H}_{\mathcal{S}}(x))-\frac{\widetilde{H}\big(H(\theta_{n})+\varepsilon_n\frac{H(\theta_n+\varepsilon_nx_{n})-H(\theta_n)}{\varepsilon_n}\big) -\widetilde{H}(H(\theta_n))}{\varepsilon_n}\Big\|_{\widetilde{\widetilde{\bE}}}.
\end{eqnarray*}
Note that by assumption $H(\theta_n) \in \widetilde{\V}_{\widetilde{H}}$ and in particular $(H(\theta_n)) \in \widetilde{\mathcal{S}}$. By the uniform quasi-Hadamard differentiability of $H$ w.r.t.\ $\mathcal{S}$ tangentially to $\bE_0\langle\bE\rangle$ with trace $\widetilde\bE$
$$
    \lim_{n \rightarrow \infty} \Big\|\frac{H(\theta_n+\varepsilon_nx_{n})-H(\theta_n)}{\varepsilon_n} - \dot{H}_{\mathcal{S}}(x)\Big\|_{\widetilde{E}}=0.
$$
Moreover $(H(\theta_n+\varepsilon_nx_{n})-H(\theta_n))/\varepsilon_n \in \widetilde{\bE}$ and $\dot{H}_{\mathcal{S}}(x) \in \widetilde{\bE}_0$, because $H$ is associated with trace $\widetilde{\bE}$ and $\dot{H}_{\mathcal{S}}(\bE_0) \subseteq \widetilde{\bE}_0$. Hence, by the uniform quasi-Hadamard differentiability of $\widetilde{H}$ w.r.t.\ $\widetilde{\mathcal{S}}$ tangentially to $\widetilde{\bE}_0\langle\widetilde{\bE}\rangle$ we obtain
$$
    \lim_{n\to\infty}\Big\| \dot{\widetilde{H}}_{\widetilde{\mathcal{S}}}(\dot{H}_{\mathcal{S}}(x))-\frac{\widetilde{H} \big(H(\theta_{n})+\varepsilon_n\frac{H(\theta_n+\varepsilon_nx_{n})-H(\theta_n)}{\varepsilon_n}\big)- \widetilde{H}(H(\theta_n))}{\varepsilon_n}\Big\|_{\widetilde{\widetilde{\bE}}}=0.
$$
This completes the proof.
\end{proof}


\end{document}